\documentclass[11pt,twoside]{article}
\usepackage{amsmath,amssymb,amsthm}
\usepackage{graphics,graphicx}
\usepackage{epstopdf}
\usepackage{lipsum}
\usepackage{braket,amsfonts,amsopn}
\usepackage[caption=false]{subfig}
\usepackage{algorithm}
\usepackage{algorithmic}
\usepackage{cases}
\usepackage{color}
\usepackage{fancyhdr}
\usepackage{empheq}
\usepackage{marginnote}
\usepackage{optidef}
\usepackage{float}
\usepackage{multicol}
\usepackage{multirow}
\usepackage{hyperref}
\usepackage{enumerate}
\usepackage{verbatim}
\usepackage{optidef}
\usepackage[asymmetric]{geometry}
\ifpdf
\DeclareGraphicsExtensions{.eps,.pdf,.png,.jpg}
\else
\DeclareGraphicsExtensions{.eps}
\fi

\newcommand{\hs}{\hspace{1em}}
\newcommand{\norm}[1]{\left\lVert#1\right\rVert}
\newcommand\AMSname{AMS subject classifications}

\newcommand\keywordsname{Key words}

\newcommand{\newcomment}[1]{}

\newenvironment{@abssec}[1]{%
	\if@twocolumn
	\section*{#1}%
	\else
	\vspace{.05in}\footnotesize
	\parindent .2in
	{\upshape\bfseries #1. }\ignorespaces
	\fi}
{\if@twocolumn\else\par\vspace{.1in}\fi}

\newenvironment{AMS}{\begin{@abssec}{\AMSname}}{\end{@abssec}}

\newtheorem{example}{Example}
\newtheorem{definition}{Definition}[section]

\numberwithin{equation}{section}
\numberwithin{figure}{section}
\numberwithin{table}{section}

\newtheorem{theorem}{Theorem}[section]

\newtheorem{proposition}{Proposition}[section]
\newtheorem{corollary}{Corollary}[section]
\newtheorem{remark}{Remark}[section]

\newcommand{\email}[1]{\protect\href{mailto:#1}{#1}}

\floatname{algorithm}{Algorithm}

\newenvironment{keywords}{\begin{@abssec}{\keywordsname}}{\end{@abssec}}

\newcommand{\mat}[1]{\left[ \begin{array}{#1} }
\newcommand{\rix}{\end{array} \right]}

\fancypagestyle{plain}{
	\fancyhf{}
	\cfoot{\thepage}
	
}

\title{2D Eigenvalue Problem I: 
Existence and Number of Solutions~\thanks{
version \today.
}
}

\author{
Yangfeng Su\footnotemark[2]
\and 
Tianyi Lu\thanks{School of Mathematical Sciences, 
Fudan University, Shanghai 200433, China
(\email{tylu17@fudan.edu.cn}, \email{yfsu@fudan.edu.cn}). 
}
\and Zhaojun Bai\thanks{Department of Computer Science
and Department of Mathematics, University of California,
Davis, CA 95616, USA (\email{zbai@ucdavis.edu})}
}

\usepackage{amsopn}
\DeclareMathOperator{\diag}{diag}

\DeclareMathOperator{\rank}{rank}
\DeclareMathOperator{\real}{Re}

\DeclareMathOperator{\Span}{span}
\DeclareMathOperator{\sign}{sign}

\ifpdf
\hypersetup{
pdftitle={2D Eigenvalue Problems II},
pdfauthor={T. Lu, Y. Su and Z. Bai}
}
\fi

\begin{document}

\maketitle

\begin{abstract}
A two dimensional eigenvalue problem (2DEVP)
of a Hermitian matrix pair $(A, C)$ is introduced in this paper.
The 2DEVP can be viewed as a linear algebra formulation of
the well-known eigenvalue optimization problem of
the parameter matrix $H(\mu) = A - \mu C$.
We present fundamental properties of the 2DEVP such as the
existence, the necessary and sufficient condition for
the finite number of 2D-eigenvalues and variational characterizations
of 2D-eigenvalues.  We use eigenvalue optimization problems from 
the minmax of two Rayleigh quotients and the computation
of distance to instability to show their connections
with the 2DEVP and new insights of these problems
derived from the properties of the 2DEVP.
\end{abstract}



\section{Introduction} \label{eq:intro}
We are interested in finding scalars $\mu, \lambda \in \mathbb{R}$
and nonzero vectors $x \in \mathbb{C}^n$ to satisfy the nonlinear equations
\begin{subequations}\label{2deig}
\begin{empheq}[left={}\empheqlbrace]{alignat=2}
(A-\mu C) x & = \lambda x,  \label{eq:1a} \\
x^HCx & = 0, \label{eq:1b}   \\
x^Hx & = 1,    \label{eq:1c}
\end{empheq}
\end{subequations}
where $A, C \in \mathbb{C}^{n\times n}$ are given Hermitian matrices and
$C$ is indefinite.
The pair $(\mu, \lambda)$ is called a \emph{2D-eigenvalue},
$x$ is called the corresponding \emph{2D-eigenvector}, and
the triplet $(\mu,\lambda,x)$ is called a \emph{2D-eigentriplet}.
We use the term ``$2D$'' based on the fact that
an eigenvalue has two components,  which is a point in
the $(\mu, \lambda)$-plane.
The nonlinear equations \eqref{2deig} are called
a \emph{two dimensional eigenvalue problem},
{2DEVP} in short,  of the matrix pair $(A,C)$.

Our interest in studying the 2DEVP \eqref{2deig}
primarily stems from eigenvalue optimization problems. 
If we regard $\mu$ as a parameter
in the 2DEVP (\ref{2deig}),
the equation \eqref{eq:1a} is a parameter eigenvalue problem
of $H(\mu) = A - \mu C$.  Since $A$ and $C$ are Hermitian,
$H(\mu)$ has $n$ real eigenvalues $\lambda_1(\mu), \lambda_2(\mu), \ldots, \lambda_n(\mu)$
for any $\mu\in\mathbb{R}$.
Suppose that these eigenvalues are sorted such that
$\lambda_1(\mu) \ge \lambda_2(\mu) \ge \cdots \ge \lambda_n(\mu)$.
When one wants to optimize an eigenvalue $\lambda_i(\mu)$ with respect
to $\mu$: 
\begin{equation} \label{eq:eigopt0}
\inf\limits_{\mu\in\mathbb{R}}\lambda_{i}(\mu),
\end{equation}
the second equation \eqref{eq:1b}
is actually a stationary condition for (local or global) maxima
or minima (see Section~\ref{sec:variationProp}).  This fact has been observed
by Overton \cite{Overton1995} when $\lambda_i(\mu_*)$ is a simple eigenvalue
of $H(\mu_*)$ at a stationary point $\mu_*$.
In general, when $\lambda_i(\mu_*)$ is a multiple eigenvalue,
to the best of our knowledge, the connection to the 2DEVP (\ref{2deig})
as presented in this paper is new.

Different equivalent conditions of
the eigenvalue optimization 
have been discovered in the literature, such as conditions
based on the generalized gradient \cite{POLAK1982} and existence of a special positive
semidefinite matrix \cite{FN1995} in the context of minimizing the largest eigenvalue
of a multivariable Hermitian matrix. These conditions will lead to a
different optimization method. 
The eigenvalue optimization in the presence of multiplicity
is still one of main challenges
\cite{FN1995,Kangal2018,KLV2017,LO1996,MYK2014,Ove1988,Overton1995,POLAK1982,SP2005}.


Blum and Chang \cite{BC1978}
considered the following so-called two-parameter or double eigenvalue problem arising
from solving a boundary value problem of ordinary differential equations with double parameters:
\begin{equation}\label{eq:tpeig}
  \left\{
  \begin{aligned}
   Ax&=\lambda C_1x+\mu C_2 x, \\
   f(x)&=0,\\
   \|x\|&=1,
  \end{aligned}
  \right.
\end{equation}
where $A, C_1, C_2 \in \mathbb{R}^{n\times n}$ and $f$ is a real-valued function.
$\lambda, \mu \in \mathbb{R}$ and $x \in \mathbb{R}^n$ are the
eigenvalues and eigenvectors to be found.
Khazanov~\cite{Kha2005}
generalized the problem (\ref{eq:tpeig}) to more than two parameters
and derived a related eigenvalue problem. 
Obviously, when $A$ and $C$ in \eqref{2deig} are real,
the 2DEVP (\ref{2deig}) is a special case of \eqref{eq:tpeig}.
The reasons for our study of the 2DEVP (\ref{2deig}) are two-fold.
There is a lack of analysis of the two-parameter eigenvalue 
problem \eqref{eq:tpeig} such as the existence, and convergence analysis 
of proposed algorithms in \cite{BC1978,Kha2005}. 
More important, only real matrices are considered.
The extension to the complex matrices is necessary for applications such as
calculating the distance to instability \cite{ahownear1985}
(see Section~\ref{applications}). It is a non-trivial extension since
one cannot take the derivatives of the equations \eqref{eq:1b} 
and \eqref{eq:1c} directly.


The objectives of this paper include revealing the relationship between
the 2DEVP \eqref{2deig} and the eigenvalue optimization of 
the matrix $H(\mu) = A - \mu C$,
and studying fundamental properties of the 2DEVP \eqref{2deig} such as the
existence and the necessary and sufficient condition for the finite number 
of 2D-eigenvalues.
This is the first paper of ours in a sequel on the 2DEVP \eqref{2deig}.
In the forthcoming work, we will show that by transforming the eigenvalue 
optimization \eqref{eq:eigopt0} into the 2DEVP \eqref{2deig},
we will be able to migrate well-established Rayleig quotient iteration 
for solving symmetric eigenvalue problems to the large-scale 
2DEVP \eqref{2deig}.

The rest of this paper is organized as follows.
In Section \ref{applications},
we discuss two eigenvalue optimizations as the origins 
of the 2DEVP \eqref{2deig}.
In Section \ref{sec_2by2}, we use simple 2-by-2 2DEVPs to reveal
some essential features and complexity of the 2DEVP.
In Section \ref{sec_parameterEVP}, we study the related
parameter eigenvalue problems and introduce the notion of
sorted eigencurves and
analyticalized eigencurves.
In Section \ref{sec:variationProp}, 
we investigate existence and variational characterization of 
2D-eigenvalues by exploiting the connection between 
the 2D-eigenvalues and the stationary points of sorted eigencurves.
In Section \ref{sec_wellposed}, we provide 
a necessary and sufficient condition
for the existence of finite number of 2D-eigenvalues.
In Section~\ref{sec_examples}, we revisit the two eigenvalue optimization
problems in Section \ref{applications} to show new insights derived from
the properties of the 2DEVP \eqref{2deig}. 
Concluding remarks are in Section \ref{sec:conclude}.



\section{Applications}\label{applications}
In this section, we discuss two eigenvalue optimization problems 
that can be formulated as the 2DEVP~\eqref{2deig}.

\subsection{Minmax of Rayleigh Quotients}\label{sec:minmax2RQs}

Given Hermitian matrices $A, B \in \mathbb{C}^{n \times n}$, 
the following minmax problem of Rayleigh quotients (RQminmax):
\begin{equation}\label{prob:minmaxRQ}
	\min_{x\neq0}\,\max \Set{ \dfrac{x^HAx}{x^Hx}, \dfrac{x^HBx}{x^Hx} },
\end{equation}
arises from quadratic constrained quadratic programs (QCQP) 
with constraints \cite{GH2013}.  

We have the following theorem for the characterization of the solution of the RQminmax~\eqref{prob:minmaxRQ}.

\begin{theorem}[\cite{GH2013}]\label{thm:classification0}  
Let $\lambda_A$ ($\lambda_B$) be the minimum eigenvalue of $A$ ($B$), 
$S_A$ ($S_B$) be an orthonormal basis of 
the corresponding eigen-subspace of $\lambda_A$ ($\lambda_B$), 
and $(\theta_A,z_A)$ ($(\theta_B,z_B)$)  
{be} the minimum eigenpair of $S_B^HAS_B$ ($S_A^HBS_A$). 
\begin{enumerate}[I.] 
\item\label{item:A1}
If $\lambda_A>\theta_B$, 
then $x_* = S_Az_B$ is a solution of the RQminmax \eqref{prob:minmaxRQ}.

\item\label{item:A2}
If $\lambda_B>\theta_A$, 
then $x_* = S_Bz_A$ is a solution of 
the RQminmax~\eqref{prob:minmaxRQ}.

\item\label{item:A3}
Otherwise, namely $\lambda_A \leq \theta_B$ and $\lambda_B \leq \theta_A$, 
$x_*$ is a solution of the RQminmax~\eqref{prob:minmaxRQ} 
if and only if 
$x_*$ is an eigenvector corresponding to the minimum eigenvalue 
$\lambda_{\min}(A-\mu_* C)$ of $A-\mu C$ and $x_*^HCx_*=0$, 
where $C = A-B$ and
$\mu_*$ is an optimizer of the following eigenvalue optimization problem 
(EVopt): 
\begin{equation}\label{eq:EVopt0}
\max\limits_{\mu\in \mathbb{R}}\lambda_{\min}(A-\mu C).
\end{equation}
\end{enumerate}
\end{theorem}

Theorem~\ref{thm:classification0}  
is actually a generalization of the results by Gaurav and Hari \cite{GH2013}. 
In \cite{GH2013}, it is implicitly assumed that the multiplicities 
of $\lambda_A$, $\lambda_B$ and $\lambda_{\min}(A-\mu_*C)$ are one, while 
Theorem~\ref{thm:classification0} does not need 
this assumption. In addition, Theorem~\ref{thm:classification0} 
provide a precise description of the relation between the 
RQminmax~\eqref{prob:minmaxRQ} and the EVopt~\eqref{eq:EVopt0}.
A proof of Theorem~\ref{thm:classification0} is in Appendix~\ref{append:thmRQ}.

In Theorem~\ref{anavarchar}, we will show that the solution of
the EVopt~\eqref{eq:EVopt0} is the minimum 2D-eigentriplet of 
the 2DEVP \eqref{2deig}.  We note that the RQminmax \eqref{prob:minmaxRQ} 
also arises from optimal conditions of CDT problems
in the trust region methods for nonlinear equality constrained optimization
\cite{Celis1985,Yuan1990}, transmit beamforming \cite{Gershman2010,Karipidis2007,Zhang2009}, 
MIMO relay optimization \cite{Roos1999} and 
cognitive radio networks \cite{Zhang2011}, and is closely related to the well-known S-lemma in control theory and robust optimization \cite{Polik2007,zong2010}.

\subsection{Distance to instability}
A basic problem in studying the stability of linear dynamical systems 
is to compute the distance to instability, 
e.g. \cite[$\mathsection49$]{Trefethen2005}.
In matrix notation, for a stable matrix 
$\widehat{A}\in \mathbb{C}^{m\times m}$,
namely all eigenvalues of $\widehat{A}$ 
locate in the left half of the complex plane $\mathbb{C}$, 
the distance to instability (DTI) $\beta(\widehat{A})$ is defined by
\begin{equation} \label{eq:defdti} 
\beta(\widehat{A})=
\min\Set{ \norm{E}\ \big|\ \widehat{A}+E \text{  is unstable}, 
E \in \mathbb{C}^{m\times m}}.
\end{equation} 
Van Loan \cite{ahownear1985} showed that the DTI $\beta(\widehat{A})$ can be 
recast as the singular value optimization 
\begin{equation} \label{eq:relationse}  
\beta(\widehat{A}) 
 = \min_{\mu \in \mathbb{R}} \sigma_{\min} (\widehat{A}-\mu {\tt i} I), 
\end{equation} 
where ${\tt i} = \sqrt{-1}$ and 
$\sigma_{\min}(X)$ refers to the smallest singular value 
of the matrix $X$.

By the relation between the singular values of a matrix $X$ 
and eigenvalues of Hermitian matrix $\begin{bmatrix}
0&X^H\\
X&0
\end{bmatrix}$, see e.g. \cite[Theorem~3.3]{demmel1997applied},  
	the singular value optimization \eqref{eq:relationse} 
can be transformed to the eigenvalue optimization (EVopt)
\begin{equation} \label{eq_distance}
	\beta(\widehat{A}) =\min_{\mu \in \mathbb{R}}\lambda_m(A-\mu C),
\end{equation} 
where $A$ and $C$ are $2m\times 2m$ matrices: 
$A = \left[ \begin{smallmatrix} 
	& \widehat{A} \\ 
	\widehat{A}^H & \\ 
\end{smallmatrix} \right]$ 
and $C = \left[ \begin{smallmatrix} 
	& {\tt i}I  \\ 
	-{\tt i}I  & \\ 
\end{smallmatrix} \right]$,  
and $\lambda_m(A-\mu C)$ is the smallest positive eigenvalue of $A-\mu C$. 

In Theorem~\ref{thm_minima_is_2devl}, we will show that if
$\mu_*$ is an optimizer of \eqref{eq_distance}, 
then $(\mu_*, \lambda_*)$ is a 2D-eigenvalue of the 2DEVP \eqref{2deig},
where $\lambda_* = \lambda_m(A-\mu_* C)$. 

\section{2-by-2 2DEVPs} \label{sec_2by2}

We start from the smallest 2-by-2 2DEVPs, namely $A$ and $C$ are 2-by-2 
Hermitian matrices to investigate the existence 
of 2D-eigentriplets, and connection with the stational points
of the eigencurves of the parameter matrix $A-\mu C$. 
Without loss of generality, we assume that the indefinite matrix $C$ is diagonal
with diagonal elements $c_1>0$ and $c_{2}<0$, and 
the 2-by-2 matrices of $A$ and $C$ are denoted by
\[
A  = \begin{bmatrix} 
a_{11} & a_{12} \\ 
\bar{a}_{12} & a_{22} \\ 
\end{bmatrix} 
\quad \mbox{and} \quad 
C  = \begin{bmatrix} 
c_{1} &   \\ 
      & c_{2} \\ 
\end{bmatrix}. 
\] 

First we note that up to a scaling
any nonzero vector $x$ satisfying \eqref{eq:1b} and \eqref{eq:1c} is of the form
\begin{equation} \label{eq:xform}
x(\alpha)=\frac{1}{\sqrt{c_1-c_2}}\begin{bmatrix}
\sqrt{-c_2}\\ \alpha\sqrt{c_1}\end{bmatrix}, 
\end{equation}
where $\alpha \in \mathbb{C}$ and $|\alpha|=1$. 
By multiplying $x^H(\alpha)$ and $x^H(\alpha)C$ on the
left of ~\eqref{eq:1a} respectively, we have
\begin{equation}\label{eq:mulambda2by2}
\mu=\dfrac{x^H(\alpha)CAx(\alpha)}{\|Cx(\alpha)\|^2}
\quad \mbox{and} \quad
\lambda=x^H(\alpha)Ax(\alpha),
\end{equation}
and the triplet $(\mu,\lambda,x(\alpha))$ satisfies \eqref{eq:1a}.

Since there exist infinitely many $\alpha$ with $|\alpha|=1$,
the 2DEVP (\ref{2deig}) seems to possess infinite number of
2D-eigenvalues.
However, this does not imply that any triplet
$(\mu, \lambda, x(\alpha))$ defined
in \eqref{eq:xform} and \eqref{eq:mulambda2by2}
is a 2D-eigentriplet of 2DEVP~\eqref{2deig}
since only {\em real} pairs $(\mu,\lambda)$
are of interest.

Obviously, $\lambda$ in \eqref{eq:mulambda2by2} is always real.
By straightforward calculation, we have
\begin{equation}\label{eq:mu2by2}
\mu=\dfrac{a_{11}
-a_{22}+(c_1\alpha a_{12}+c_2\overline{\alpha a_{12}})/\sqrt{-c_1c_2}}
{c_1-c_2}.
\end{equation}
Since $c_1>0$ and $c_2<0$, $\mu$ is real if and only if $\alpha a_{12}$ is real.
There are two cases:
\begin{itemize}
\item $a_{12}\not=0$. By choosing
$\alpha_{1,2}=\pm |a_{12}|/a_{12}$, then we have
\begin{align} \label{eq:mu22}
\mu_{1,2}= \dfrac{a_{11}-a_{22}\pm|a_{12}|(c_1+c_2)/\sqrt{-c_1c_2}}{c_1-c_2},
\end{align}
and
\begin{equation} \label{eq:lambda22}
\lambda_{1,2}
=\dfrac{a_{11}/c_1 -a_{22}/c_2 \pm 2|a_{12}|/\sqrt{-c_1 c_2}}{(c_1-c_2)/(-c_1 c_2)}.
\end{equation}
and
$$
x_{1,2} = x(\alpha_{1,2}).
$$
Therefore, the 2DEVP~\eqref{2deig}
has exactly two 2D-eigentriplets $(\mu_1,\lambda_1,x_1)$
and $(\mu_2,\lambda_2,x_2)$.
In addition, we note that
$\lambda_1$ and $\lambda_2$ are
simple eigenvalues of $A-\mu_1 C$ and $A-\mu_2 C$, respectively.

\item $a_{12}=0$.
The 2D-eigentriplets are given by
\begin{equation} \label{eq:mulambda22}
(\mu_1, \lambda_1, x_1(\alpha)) \equiv
\left(\dfrac{a_{11}-a_{22}}{c_1- c_2}, \
\dfrac{a_{22}c_1-a_{11}c_2}{c_1-c_2}, \
x(\alpha) \right)
\end{equation}
for any $\alpha \in \mathbb{C}$ with $|\alpha|=1$.
There are infinitely many eigentriplets with
the same 2D-eigenvalue.
In addition, $\lambda_1$ is an eigenvalue of $A-\mu_1 C$ 
with multiplicity 2.
\end{itemize}

If we discard the equation~\eqref{eq:1b}, the 2DEVP
becomes a parameter eigenvalue problem
of $H(\mu) = A-\mu C$. There are also two cases:
\begin{itemize}
\item $a_{12}\not=0$.  In this case, we have two smooth eigencurves
$\lambda_1(\mu)$ and $\lambda_2(\mu)$ of $H(\mu)$:
\[
\lambda_{1,2}(\mu) =
\frac{1}{2}\left(
{a_{11}+a_{22}-\mu(c_1+c_2)\pm
\sqrt{[(a_{11}-a_{22})-\mu(c_1-c_2)]^2+4\left|a_{12}\right|^2}} \right).
\]
By setting $\lambda^{\prime}_{1,2}(\mu) = 0$, we have the following
two stationary points of eigencurves $\lambda_1(\mu)$ and $\lambda_2(\mu)$:
\begin{align*}
{\mu}_{1,2} & = \dfrac{a_{11}-a_{22}\pm |a_{12}|(c_1+c_2)
/\sqrt{-c_1c_2}}{c_1-c_2}.
\end{align*}
By \eqref{eq:mu22}, $({\mu}_1,\lambda_1({\mu}_1))$
and $({\mu}_2,\lambda_2({\mu}_2))$ are 2D-eigenvalues. 

\item $a_{12}=0$. In this case, eigencurves 
$\lambda_1(\mu)$ and $\lambda_2(\mu)$ are
\[
\lambda_{1,2}(\mu) =
\frac{1}{2}\left( {a_{11}+a_{22}-\mu(c_1+c_2)\pm
\left|a_{11}-a_{22}-\mu(c_1-c_2)\right|} \right).
\]
Moreover, $\lambda_1(\mu)$ and $\lambda_2(\mu)$
intersect and are not differentiable
at the intersection point
${\mu}_{1} = {\mu}_{2} =  \dfrac{a_{11}-a_{22}}{c_1-c_2}$.
Furthermore, since
$\left|c_1+c_2\right| < \left|c_1-c_2\right|$,
${\mu}_1$ and ${\mu}_2$
are the minimum and maximum of
eigencurves $\lambda_1(\mu)$ and
$\lambda_2(\mu)$, respectively.
By \eqref{eq:mulambda22},
$({\mu}_1,\lambda_1({\mu}_1))$
and
$({\mu}_2,\lambda_2({\mu}_2))$
are the 2D-eigenvalues. 
\end{itemize}


\begin{example}\label{2by2examplePlot}
{\rm
Consider the 2DEVP \eqref{2deig} with
\[
 A = \begin{bmatrix}
     1 & a_{12} \\
     a_{12} & 1
   \end{bmatrix},\quad
 C = \begin{bmatrix}
       0.2 & 0 \\
       0 & -0.5
     \end{bmatrix}.
\]
where $a_{12} \in \mathbb{R}$. 
Figure~\ref{figTwobytwo} shows
the eigenfunctions $\lambda_1(\mu)$ and $\lambda_2(\mu)$
of $A-\mu C$ with $a_{12} \neq 0$ and $a_{12} =0$, 
where $\lambda_1(\mu)$ and $\lambda_2(\mu)$ are sorted such that 
$\lambda_1(\mu)\geq\lambda_2(\mu)$.
The 2D-eigenvalues $(\mu_{\ast}, \lambda_{\ast})$
are marked by red circles. We observe that
\begin{enumerate}
\item in the case $a_{12} \neq 0$,
$\lambda_1(\mu)$ and $\lambda_2(\mu)$ are
differentiable and the 2D-eigenvalues correspond to 
stationary points of $\lambda_1(\mu)$ and $\lambda_2(\mu)$,
see Figure~\ref{figTwobytwo}(a), 

\item in the case $a_{12}= 0$, $\lambda_1(\mu)$ and $\lambda_2(\mu)$
are not differentiable at $\mu_*$, but the 2D-eigenvalue still
corresponds to minimum and maximum of $\lambda_1(\mu)$ and $\lambda_2(\mu)$,
see Figure~\ref{figTwobytwo}(b). 
\end{enumerate}
In addition, we note that if $\mu$ is not restricted to be real, then 
the red line in Figure~\ref{figTwobytwo}(a) is a plot $(\real(\mu), \lambda)$
calculated from \eqref{eq:mulambda2by2} with $\mu \in \mathbb{C}$. 
Therefore, if $\mu$ is not restricted to be real, the 2DEVP has 
continuous spectrum.
\hfill $\Box$

\begin{figure}[tbhp]
\centering
   \subfloat[The case $a_{12} = 0.2$]{
   \includegraphics[width=0.45\textwidth]{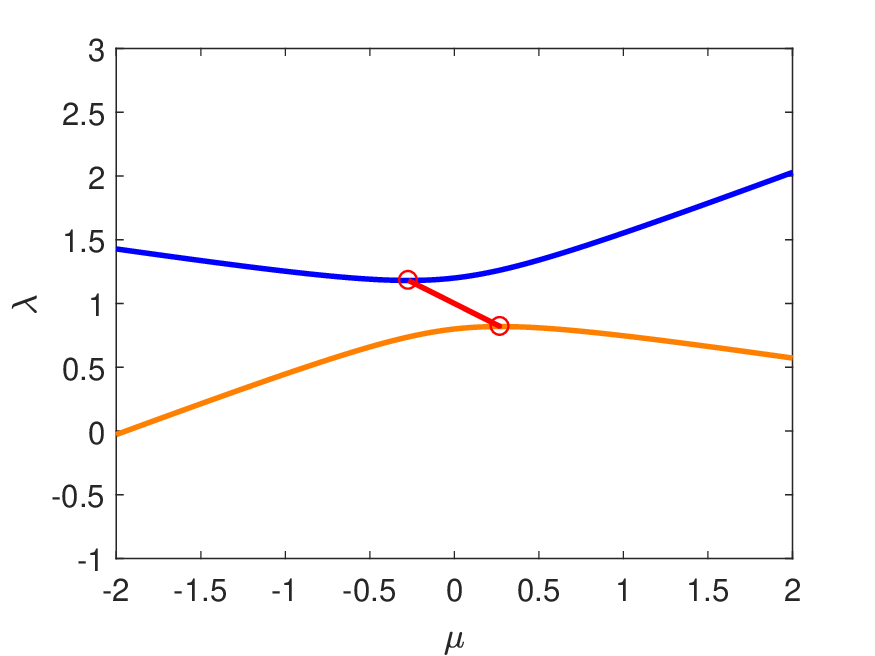}
   \label{fig:a}
   }\hs
  \subfloat[The case $a_{12}= 0$]{
   \includegraphics[width=0.45\textwidth]{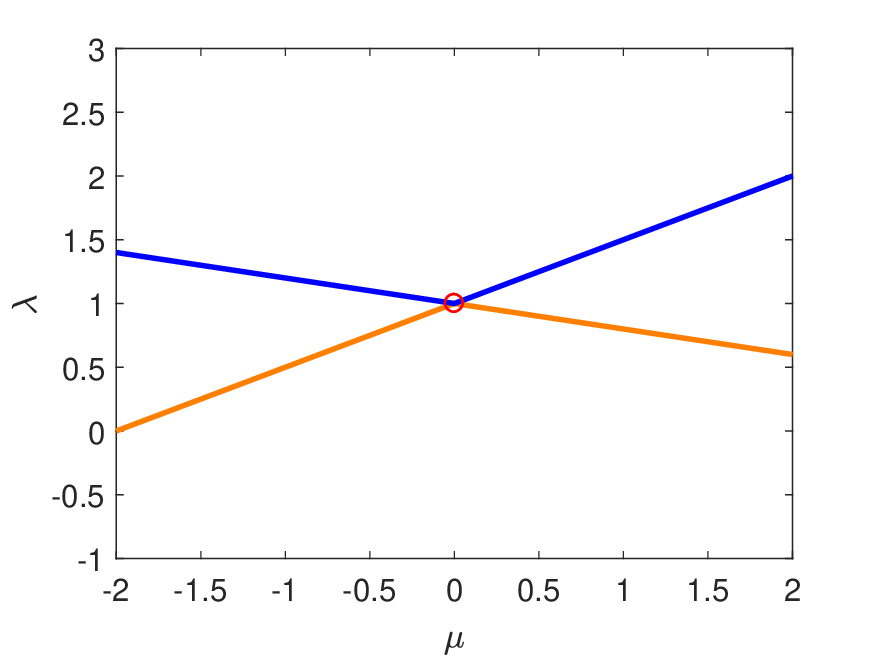}
   \label{fig:b}
   }
  \caption{
Examples of 2-by-2 2DEVPs.
The blue and orange curves are
eigenvalues $\lambda_1(\mu)$ and $\lambda_2(\mu)$
of $A-\mu C$ sorted such that
$\lambda_1(\mu)\geq\lambda_2(\mu)$.
The 2D-eigenvalues $(\mu_{\ast}, \lambda_{\ast})$
are marked by red circles.
} \label{figTwobytwo}
\end{figure}

} \end{example}

\section{The associated parameter
eigenvalue problem} \label{sec_parameterEVP}

If we discard the equation \eqref{eq:1b},
the remaining two equations of the 2DEVP~\eqref{2deig}
are a parameter eigenvalue problem of
$H(\mu) = A-\mu C$ with real parameter $\mu$.
In this section, we study
the properties of this associated parameter
eigenvalue problem.

For $\mu\in\mathbb{R}$, there exist $n$ real eigenvalues $\lambda_i(\mu)$ and
corresponding orthonormal eigenvectors $x_i(\mu)$ of $A-\mu C$. 
If $\lambda_i(\mu)$ are sorted 
such that $\lambda_1(\mu)\ge \cdots \ge \lambda_n(\mu)$, then we
have $n$ {\em sorted eigencurves} $\lambda_i(\mu)$ for $i = 1, 2, \ldots, n$.
The following theorem is a well-known result on 
the convexity of the extreme eigencurves 
$\lambda_1(\mu)$ and $\lambda_n(\mu)$, 
see e.g. \cite{Fan1949,Lewis1996,Overton1991}.

\begin{theorem} \label{thm_convex}
$\lambda_{1}(\mu)$ is convex and $\lambda_{n}(\mu)$ is concave.
\end{theorem}


The sorted eigencurves $\lambda_i(\mu)$ are
continuous and may be non-differentiable on intersections,
see Figure~\ref{fig_reordered}(a).
The following theorem is a direct application of
Theorem S6.3 in \cite{Gohberg2009Matrix} and shows that
with proper reordering, the eigencurves $\lambda_i(\mu)$
can be analyticalized.

\begin{theorem}[\cite{Gohberg2009Matrix}]\label{analytic_prop}
For Hermitian matrices $A$ and $C$,
there exist scalar functions $\lambda_1(\mu)$, $\cdots$, $\lambda_n(\mu)$
and matrix-valued functions
$X(\mu)=\begin{bmatrix}x_1(\mu),& \cdots, & x_n(\mu)\end{bmatrix}$ 
such that for $\mu\in\mathbb{R}$,
\begin{equation}\label{analytic_eigenfuncs}
\begin{aligned}
A-\mu C & =X(\mu)
\diag\begin{bmatrix}\lambda_1(\mu),&\cdots,&\lambda_n(\mu)\end{bmatrix}X^H(\mu), \\
X^H(\mu)X(\mu) &=I,
\end{aligned}
\end{equation}
and furthermore, $\lambda_i(\mu)$ and $x_i(\mu)$ are real analytic on $\mu\in\mathbb{R}$.
\end{theorem}

By \eqref{analytic_eigenfuncs}, 
$(\lambda_i(\mu),x_i(\mu))$ are eigenpairs of $A-\mu C$.
The real analytic eigencurves $\lambda_i(\mu)$ in Theorem~\ref{analytic_prop} will be called
{\em analyticalized eigencurves}. 
Analyticalized eigencurves may be different from
sorted eigencurves as illustrated in Figure~\ref{fig_reordered}.
For clarification, in the rest of the paper, we use
$\lambda_i(\mu)$ to denote a sorted eigencurve of
$A-\mu C$, and $\widetilde{\lambda}_i(\mu)$ to denote
an analyticalized eigencurve of $A-\mu C$.

\begin{figure}[tbhp]
\centering
\subfloat[Sorted eigencurves]{
\includegraphics[width=0.45\textwidth]{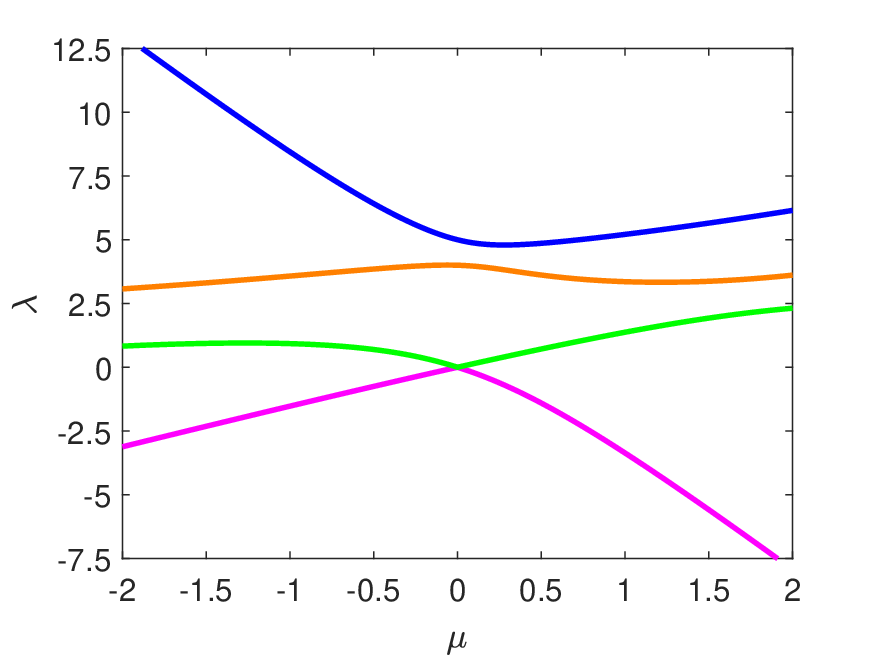}}\hs
\subfloat[Analyticalized eigencurves]{
\includegraphics[width=0.45\textwidth]{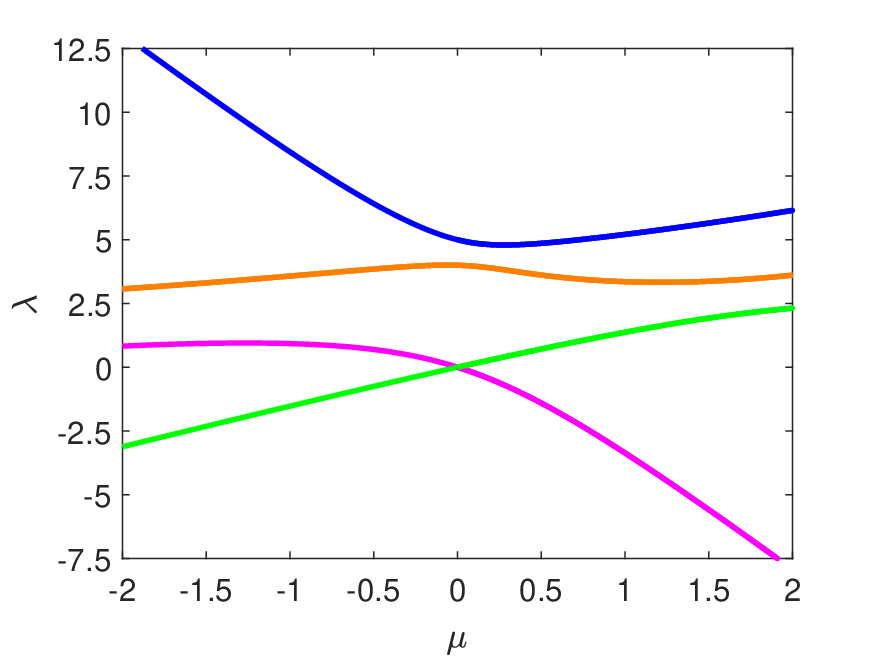}}
\caption{Illustration of eigencurves.}
\label{fig_reordered}
\end{figure}

Recall that a function $f(\mu)$ is called real analytic \cite[p.\,3, Definition~1.1.5]{Krantz2002} on an open set $U\subseteq\mathbb{R}$, if and only if $f(\mu)$ can be expanded into a convergent power series in a neighbour centered at any $\mu_0\in\mathbb{U}$.
The following theorem lists some properties 
of a real analytic function, see \cite[pp.~4, 9, 10, 14, 19]{Krantz2002}
for details.

\begin{theorem}[\cite{Knopp1990theoryI}]
\label{analytic_properties}
Let $f(\mu)$ be real analytic on $U$. Then
\begin{enumerate}[(i)]

  \item $f\in\mathrm{C}^{\infty}(U)$.
  \item\label{item:product} The sum and product of two real analytic functions on $U$ are also real analytic on $U$.
  \item\label{item:fprime} $f'(\mu)$ is also real analytic on $U$.
  Furthermore, if $f(\mu)=\sum_{i=0}^{+\infty}\alpha_i(\mu-\mu_0)^i$ is 
the power series expansion of $f$ at $\mu_0$, then 
$f'(\mu)=\sum_{i=1}^{+\infty}i\alpha_i(\mu-\mu_0)^{i-1}$
is the power series expansion of $f'$ with the same radius of convergence.
\item\label{item:uniquess} Assume $U$ is an interval. If $g(\mu)$ is also real analytic on $U$, and there exists a sequence $\{\mu_k\}_{k=1}^{\infty}\subseteq U$, where $u_k$ are distinct and $\lim\limits_{k\rightarrow\infty}\mu_k=\mu_*\in U$ such that $f(\mu_k)=g(\mu_k)$, then $f \equiv g$ on $U$. This is known as the identity theorem for analytic functions.
\item\label{item:component} Assume $U$ is an open interval on $\mathbb{R}$, and $f$ takes values in an open interval $V\subseteq \mathbb{R}$. Let $g$ be real analytic on $V$. Then $g\circ f$ is real analytic on $U$.
\end{enumerate}
\end{theorem}

One important benefit of introducing analyticalized eigencurves
$\widetilde{\lambda}_i(\mu)$ is that we can
have the notion of derivatives of the sorted eigencurves 
at any point $\mu \in \mathbb{R}$, even the one corresponds 
to the intersection of the sorted eigencurves.  
To that end, let us first present the following theorem derived from 
Theorem~1 of \cite[p.45]{Rel1969} 
to show that the derivatives of analyticalized eigencurves 
$\widetilde{\lambda}_i(\mu)$ can be calculated through 
solving an eigenvalue problem. 

%
%

\begin{theorem}[\cite{Rel1969}] \label{analytic_prop3}
Let $\widetilde{\lambda}_1(\mu),\cdots,\widetilde{\lambda}_n(\mu)$ be
the analyticalized eigencurves of $A-\mu C$.
Assume $\lambda_0$ is an eigenvalue of $A-\mu_0C$ with algebraic
multiplicity $k$, i.e., $\widetilde{\lambda}_j(\mu_0)=\lambda_0$
for $p\leq j\leq p+k-1$ with some integer $p \geq 1$.
Let $X_k$ be an orthonormal eigen-subspace corresponding to 
$\lambda_0$. Then by counting multiplicities, $\widetilde{\lambda}'_j(\mu_0)$ 
has one-to-one correspondence with the eigenvalues of $-X_k^HCX_k$ 
for $p\leq j\leq p+k-1$. 
\end{theorem} 
\begin{proof}
For the sake of completeness, we provide a proof here.
We begin with the following identity: 
\begin{equation}
(A-\mu_0C-\lambda_0I)x_j(\mu)-(\mu-\mu_0) Cx_j(\mu)-
(\widetilde{\lambda}_j(\mu)-\lambda_0)x_j(\mu)=0
\end{equation}
for $p\leq j\leq p+k-1$.
Multiplying $x^H_i(\mu_0)$ on the left and dividing by $\mu-\mu_0$, 
we have
\[
-x^H_i(\mu_0)Cx_j(\mu)=
\frac{\widetilde{\lambda}_j(\mu)-\lambda_0}{\mu-\mu_0}x^H_i(\mu_0)x_j(\mu).
\]
Let $\mu\rightarrow\mu_0$, we obtain
\begin{equation}\label{midequ}
-x^H_i(\mu_0)Cx_j(\mu_0)=\widetilde{\lambda}_j'(\mu_0)x^H_i(\mu_0)x_j(\mu_0).
\end{equation}
Since $\{x_i(\mu_0)\}_{i=p}^{p+k-1}$ is a basis of $X_k$, 
\eqref{midequ} is equivalent to
\begin{equation} \label{eq:XC} 
-X_k^HCx_j(\mu_0)=\widetilde{\lambda}_j'(\mu_0)X_k^Hx_j(\mu_0).
\end{equation} 
Furthermore, there exists $z_j$, such that $x_j(\mu_0)=X_kz_j$ 
and $\{z_j\}_{j=p}^{p+k-1}$ are orthonormal. 
Thus the equation \eqref{eq:XC} is turned into
\[
-X_k^HCX_kz_j=\widetilde{\lambda}_j'(\mu_0)z_j.
\]
This implies that $\widetilde{\lambda}_j'(\mu_0)$
for $p\leq j\leq p+k-1$ have
one-to-one correspondence with eigenvalues of
$-X_k^HCX_k$, counting multiplicities.
\end{proof}

The following theorem is an application of Theorem~\ref{analytic_prop3}
to show that the sorted eigencurves $\lambda_j(\mu)$ have 
well-defined one-sided derivatives.
	\begin{theorem}\label{thm:onesidedderivative}
	Assume $(\mu_*,\lambda_*)$
	is an intersection of $k$ sorted eigencurves, i.e., 
	$\lambda_j(\mu_*)=\lambda_*$ for $p\leq j\leq p+k-1$ for some 
	integer $p\geq 1$, and $\lambda_{p-1}(\mu_*)>\lambda_*>\lambda_{p+k}(\mu_*)$. 
	Let $X_k$ be an orthonormal eigen-subspace of 
	the eigenvalue $\lambda_*$ of $A-\mu_* C$.
	Then for $p\leq j\leq p+k-1$, the one-sided derivatives 
	\[
	\lambda^{'(+)}_{j}(\mu_*)
	\equiv  \lim\limits_{t\rightarrow 0^+}
	\frac{\lambda_j(\mu_*+t)-\lambda_j(\mu_*)}{t}
	\quad \mbox{and} \quad
	\lambda^{'(-)}_{j}(\mu_*)
	\equiv  \lim\limits_{t\rightarrow 0^-}
	\frac{\lambda_j(\mu_*+t)-\lambda_j(\mu_*)}{t},
	\]
	exist. Furthermore, both multisets 
	$\{\lambda^{'(-)}_{j}(\mu_*) \mid p\leq j\leq p+k-1\}$ and 
	$\{\lambda^{'(+)}_{j}(\mu_*)\ |\ p\leq j\leq p+k-1\}$ have 
	one-to-one correspondence with the multiset of eigenvalues of 
	$-C_k \equiv -X^H_k C X_k$, i.e, if the eigenvalues of 
	$-C_k \equiv -X^H_k C X_k$ 
	are $\tau_1\geq\tau_2\geq\cdots\geq\tau_k$, then 
	$\lambda^{'(+)}_{p-1+i}(\mu_*)
	=\tau_{i}=\lambda^{'(-)}_{p+k-i}(\mu_*)$ for $i=1,\cdots,k$.
\end{theorem}
\begin{proof} 
	We first prove by contradiction that there exists $r>0$, such that
	in the interval $(\mu_*, \mu_*+r)$, for any $i,j$, 
	there are only following two cases between any two 
	analyticalized eigencurves 
	$\widetilde{\lambda}_i(\mu)$ and $\widetilde{\lambda}_j(\mu)$ 
	of $A-\mu C$: 
	\[
	\widetilde{\lambda}_i(\mu)=\widetilde{\lambda}_j(\mu)
	\quad \mbox{or} \quad 
	\widetilde{\lambda}_i(\mu)\neq\widetilde{\lambda}_j(\mu)
	\]
	for any $\mu\in(\mu_*, \mu_*+r)$. 
	If $r$ does not exist, then we can find a fixed pair 
	$(i,j)$ and a sequence $\{\mu_m\}_{m=1}^{\infty}$
	such that $\widetilde{\lambda}_i(\mu_m)=\widetilde{\lambda}_j(\mu_m)$,
	$\mu_m\rightarrow\mu_*$, $\mu_m\neq\mu_*$ but
	$\widetilde{\lambda}_i(\mu)\not\equiv\widetilde{\lambda}_j(\mu)$,
	which contradicts Theorem~\ref{analytic_properties}\eqref{item:uniquess}. 
	
	We next prove that in the interval $[\mu_*, \mu_*+r)$, 
	each sorted eigencurve identically equals to an analyticalized eigencurve. 
	In fact, we have proved that in the interval $(\mu_*, \mu_*+r)$, 
	two analyticalized eigencurves that do not equal identically will not 
	intersect. Then by the continuity of analyticalized eigencurves, 
	for any $i,j$, there are exactly the following three cases 
	that hold for all $\mu\in(\mu_*, \mu_*+r)$:
	\[
	\widetilde{\lambda}_i(\mu)<\widetilde{\lambda}_j(\mu)
	\quad \mbox{or} \quad 
	\widetilde{\lambda}_i(\mu)=\widetilde{\lambda}_j(\mu)
	\quad \mbox{or} \quad 
	\widetilde{\lambda}_i(\mu)>\widetilde{\lambda}_j(\mu).
	\]
	This implies in the interval $(\mu_*, \mu_*+r)$, the algebraic order 
	of the analyticalized eigencurves are preserved. Thus we can 
	find a permutation $\{\ell_1,\ell_2,\cdots,\ell_n\}$ of 
	$\{1,2,\cdots,n\}$, such that 
	$\widetilde{\lambda}_{\ell_i}(\mu) = \lambda_i(\mu)$ 
	for $\mu\in(\mu_*,\mu_*+r)$ and $ i=1,\cdots,n$. By continuity, 
	$\widetilde{\lambda}_{\ell_i}(\mu_*) = \lambda_i(\mu_*)$, 
	for $i=1,\cdots,n$. Consequently, for $p\leq j\leq p+k-1$, 
	the limit
	$$
	\lim\limits_{t\rightarrow 0^+}
	\frac{\lambda_j(\mu_*+t)-\lambda_j(\mu_*)}{t}
	$$
	exist and equals to $\widetilde{\lambda}_{\ell_j}'(\mu_*)$. By Theorem~\ref{analytic_prop3}, the multiset $\{ \lambda^{'(+)}_{j}(\mu_*) \ \big|\ p\leq j\leq p+k-1 \}$ have one-to-one correspondence with the multiset of eigenvalues of $-C_k$. By a similar argument, we can show that the limit $\lambda^{'(-)}_{j}(\mu_*)$ exists and has one-to-one correspondence with the eigenvalues of $-C_k$, counting multiplicities.
	
	Furthermore, note that for $t>0$ and $p\leq j\leq p+k-2$,
	\[
	\frac{\lambda_j(\mu_*+t)-\lambda_j(\mu_*)}{t}
	\geq\frac{\lambda_{j+1}(\mu_*+t)-\lambda_{j+1}(\mu_*)}{t}
	\]
	and
	\[ 
	\frac{\lambda_j(\mu_*-t)-\lambda_j(\mu_*)}{-t}\leq
	\frac{\lambda_{j+1}(\mu_*-t)-\lambda_{j+1}(\mu_*)}{-t}.
	\]
	Thus for $p\leq j\leq p+k-2$, 
	\begin{equation}\label{eq:onesidedderivativeComp}
		\lambda^{'(+)}_{j}(\mu_*)\geq \lambda^{'(+)}_{j+1}(\mu_*)
		\quad \mbox{and} \quad 
		\lambda^{'(-)}_{j}(\mu_*)\leq \lambda^{'(-)}_{j+1}(\mu_*).
	\end{equation}
	Then the equality 
	$\lambda^{'(+)}_{p-1+i}(\mu_*) =\tau_{i}=\lambda^{'(-)}_{p+k-i}(\mu_*)$ 
	follows from the equation \eqref{eq:onesidedderivativeComp} 
	and the correspondence between 
	$\{ \lambda^{'(+)}_{j}(\mu_*) \ \big|\ p\leq j\leq p+k-1 \}$, 
	$\{ \lambda^{'(-)}_{j}(\mu_*) \ \big|\ p\leq j\leq p+k-1 \}$ 
	and $\{ \tau_i \ \big|\ 1\leq i\leq k \}$.
\end{proof}

When $k = 1$ in Theorem~\ref{thm:onesidedderivative}, 
the argument in the proof shows the following result.  

\begin{corollary}\label{cor:simple}
	Under the notation in Theorem~\ref{thm:onesidedderivative}, 
	if $k=1$, $\lambda_p(\mu)$ is differentiable at $\mu_*$ 
	and satisfies $\lambda_p'(\mu_*) = -x_p(\mu_*)^HCx_p(\mu_*)$, 
	where $x_p(\mu_*)$ is the corresponding unit eigenvector of $\lambda_p(\mu_*)$.
\end{corollary}

\section{Existence and variational characterization of 2D-eigenvalues} 
\label{sec:variationProp}

In this section, we discuss the existence of the 2D-eigenvalues
and their variational characterizations. The following theorem characterizes when $(\mu_*,\lambda_*)$ is a 2D eigenvalue.
\begin{theorem} \label{Thm:what_is_2devl}
	Let $(\mu_*,\lambda_*)$ be an intersection of $k$ analyticalized eigencurves $\widetilde{\lambda}_i(\mu)$ of $A-\mu C$, $i=1,\cdots,k$. 
	\begin{itemize}
		\item If $k=1$, then $(\mu_*,\lambda_*)$ is a 2D eigenvalue if and only if $\widetilde{\lambda}_1'(\mu_*)=0$;
		\item If $k>1$, then $(\mu_*,\lambda_*)$ is a 2D eigenvalue if and only if there exist $i,j\leq k$, such that $\widetilde{\lambda}_i'(\mu_*)\widetilde{\lambda}_j'(\mu_*)\leq0$.
	\end{itemize}
\end{theorem}

\begin{proof}
	Denote $X_k$ as the orthonormal basis of the eigen-subspace of the eigenvalue $\lambda_*$ of $A-\mu_*C$, and $C_k=X_k^HCX_k$.
	According to \eqref{2deig}, $(\mu_*,\lambda_*)$ is a 2D eigenvalue if and only if there exists nonzero vector $x_*\in\Span\{X_k\}$, such that $x_*^HCx_*=0$. Denote $x_*=X_kz$, then $x_*^HCx_*=0$ becomes $z^HC_kz=0$. Thus $(\mu_*,\lambda_*)$ is a 2D eigenvalue if and only if $-C_k$ is not a definite matrix.
	
	When $k=1$, $-C_k$ is not a definite matrix if and only if $C_k=0$. According to Theorem~\ref{analytic_prop3}, this is equivalent to $\widetilde{\lambda}_1'(\mu_*)=0$.
	
	When $k>1$, $-C_k$ is not a definite matrix if and only if $-C_k$ has both negative and positive eigenvalues or has eigenvalue 0. According to Theorem~\ref{analytic_prop3}, this is equivalent to there existing $i,j\leq k$, such that  $\widetilde{\lambda}_i'(\mu_*)\widetilde{\lambda}_j'(\mu_*)\leq0$.
\end{proof}

	\begin{theorem}\label{thm_minima_is_2devl}
	If $(\mu_*,\lambda_*)$ is a local minimum or maximum of
	a sorted eigencurve $\lambda(\mu)$ of $A-\mu C$, 
	then $(\mu_*,\lambda_*)$ must be a 2D-eigenvalue of $(A,C)$.
\end{theorem}
\begin{proof} 
		We prove the case that 
	$(\mu_*,\lambda_*)$ is a local maximum of some sorted eigencurve. The proof for the local minimum is similar.
	Assume $(\mu_*,\lambda_*)$
	is an intersection of $k$ sorted eigencurves
	$\lambda_j(\mu)$ of $A-\mu C$ for $p\leq j\leq p+k-1$ 
	with some integer $p \geq 1$. Let $\widetilde{\lambda}_1,\cdots,\widetilde{\lambda}_{k}$ be $k$ analyticalized eigencurves that satisfies $\widetilde{\lambda}(\mu_*)=\lambda_*$. Then according to Theorem~\ref{analytic_prop3} and Theorem~\ref{thm:onesidedderivative}, both multisets $\{\lambda^{'(-)}_{j}(\mu_*) \mid p\leq j\leq p+k-1\}$ and $\{\lambda^{'(+)}_{j}(\mu_*) \mid p\leq j\leq p+k-1\}$ have one-to-one correspondence with the multiset $\{\widetilde{\lambda}^{'}_{i}(\mu_*) \mid 1\leq i\leq k \}$. 
	
	Since $(\mu_*,\lambda_*)$ is a local maximum, we have 
	$\lambda^{'(+)}_{p+k-1}(\mu_*)\leq0$ and
	$\lambda^{'(-)}_{p+k-1}(\mu_*)\geq0$. By the one-to-one correspondence, $\{\widetilde{\lambda}^{'}_{i}(\mu_*) \mid 1\leq i\leq k \}$ has both non-negative and non-positive element. This implies $\widetilde{\lambda}^{'}_1=0$ when $k=1$, and there exist $i,j\leq k$ such that $\widetilde{\lambda}_i'(\mu_*)\widetilde{\lambda}_j'(\mu_*)\leq0$ when $k>1$. By Theorem~\ref{Thm:what_is_2devl}, $(\mu_*,\lambda_*)$ is a 2D eigenvalue.
	This completes the proof. 
\end{proof} 


\begin{remark}\label{re:stationary}
	{\rm 
		The proof of Theorem~\ref{thm_minima_is_2devl} is algebraic. 
		An alternative proof is to use Clarke's generalized directional 
		derivative and generalized gradient in nonsmooth 
		optimization~\cite[p.10]{Clarke1990}.
		%
		%
		%
		Specifically, if $(\mu_*,\lambda_*)$ is a stationary point 
		(locally minimum or maximum) of some sorted eigencurve $\lambda_j(\mu)$, 
		then we have the first-order optimality condition   
		\begin{equation} \label{eq:clarkecond}
			0\in\partial\lambda_j(\mu_*), 
		\end{equation}
		where $\partial\lambda_j(\mu_*)$ is Clarke's generalized 
		derivative $\partial\lambda_j(\mu)$ at $\mu_*$ of the eigencurve 
		$\lambda_j(\mu)$ \cite[p.38, Proposition~2.3.2]{Clarke1990}.
		Based on 
		Clarke's generalized derivatives of spectral functions \cite[p.585]{Lewis1996}
		and the chain rule \cite[p.42, Theorem~2.3.9]{Clarke1990}, we can derive that
		\begin{equation} \label{eq:clarkder}
			\partial\lambda_j(\mu_*) \subseteq \{ -x^H C x \mid 
			\mbox{$x$ is a unit eigenvector corresponding of $\lambda_*$ of $A-\mu_* C$}.
			\}
		\end{equation} 
		Consequently, by \eqref{eq:clarkecond} and \eqref{eq:clarkder}, we 
		conclude that if  $(\mu_*,\lambda_*)$ is a stationary point
		of some sorted eigencurve $\lambda_j(\mu)$, then there exists 
		a unit eigenvector $x_*$ corresponding to the eigenvalue 
		$\lambda_*$ of $A-\mu_*C$, such that $0 = x_*^HCx_*$. 
		Thus $(\mu_*,\lambda_*,x_*)$ is a 2D-eigentriplet of 
		the 2DEVP~\eqref{2deig}.
}\end{remark}

Theorem~\ref{thm_minima_is_2devl} shows that
if $(\mu_*,\lambda_*)$ is a local minimim or maximum of
some sorted eigencurve $\lambda(\mu)$, then
$(\mu_*,\lambda_*)$ must be a 2D-eigenvalue.
Conversely, a 2D-eigenvalue $(\mu, \lambda)$ does not
necessarily correspond to a local minimum or maximum of
a sorted eigencurve $\lambda(\mu)$ as shown in the following
example. 

\begin{example} \label{example1}
{\rm
Let
\[
A_1=\begin{bmatrix} 2 & 0 & 1 \\ 0 & 0 & 1 \\ 1 & 1 & 0 \end{bmatrix}, \quad
C_1=\begin{bmatrix} 1& 0 & 1 \\ 0 & 1 & 1\\ 1 & 1 & 0 \end{bmatrix}.
\]
The three sorted eigencurves
$\lambda_1(\mu) \geq \lambda_2(\mu) \geq \lambda_3(\mu)$
of $A_1 - \mu C_1$ are depicted in Figure~\ref{FigAllevls_singularcase1}(a) 
in blue, red and yellow, respectively.
$(\mu,\lambda,x) = (1, 0, e_3)$ is a 2D-eigentriplet and
the 2D-eigenvalue $(\mu,\lambda) = (1, 0)$ is on the 
eigencurve $\lambda_2(\mu)$. However, it is neither
a local minimum nor a maximum of $\lambda_2(\mu)$ as shown
in the close up plot in Figure~\ref{FigAllevls_singularcase1}(b).

Let
\[
A_2=\begin{bmatrix}
2 & 0 & 1 & 0\\ 0 & 0 & 1 & 0\\ 1 & 1 & 0 & 0 \\0&0&0&0.6
\end{bmatrix},
\quad
C_2=\begin{bmatrix}
1& 0 & 1 & 0\\ 0 & 1 & 1 & 0\\ 1 & 1 & 0 & 0\\0&0&0&0.6
\end{bmatrix}.
\]
The four sorted eigencurves of $A_2 - \mu C_2$
are depicted in Figure~\ref{FigAllevls_singularcase1}(c).
Note that 
\[ 
\mbox{det}\left(A_2-\mu C_2-\lambda I\right)
=(0.6-0.6\mu-\lambda)\det\left(A_1-\mu C_1-\lambda I\right).
\] 
Thus three analyticalized eigencurves of $A_2 - \mu C_2$ 
are identical to ones of $A_1 - \mu C_1$. 
The 2D-eigenvalue $(\mu,\lambda) = (1, 0)$ of $(A_2, C_2)$
corresponds to the intersection of eigencurves 
$\lambda_2(\mu)$ and $\lambda_3(\mu)$. However,
it is neither a local minimum nor a maximum of
$\lambda_2(\mu)$ and $\lambda_3(\mu)$ as shown in
the closed up plot in Figure~\ref{FigAllevls_singularcase1}(d).
\hfill $\Box$
} \end{example}


\begin{figure}[tbhp]
\centering
\subfloat[three (sorted) eigencurves $\lambda_1(\mu) \geq \lambda_2(\mu) \geq \lambda_3(\mu)$]
{\includegraphics[width=0.45\textwidth]{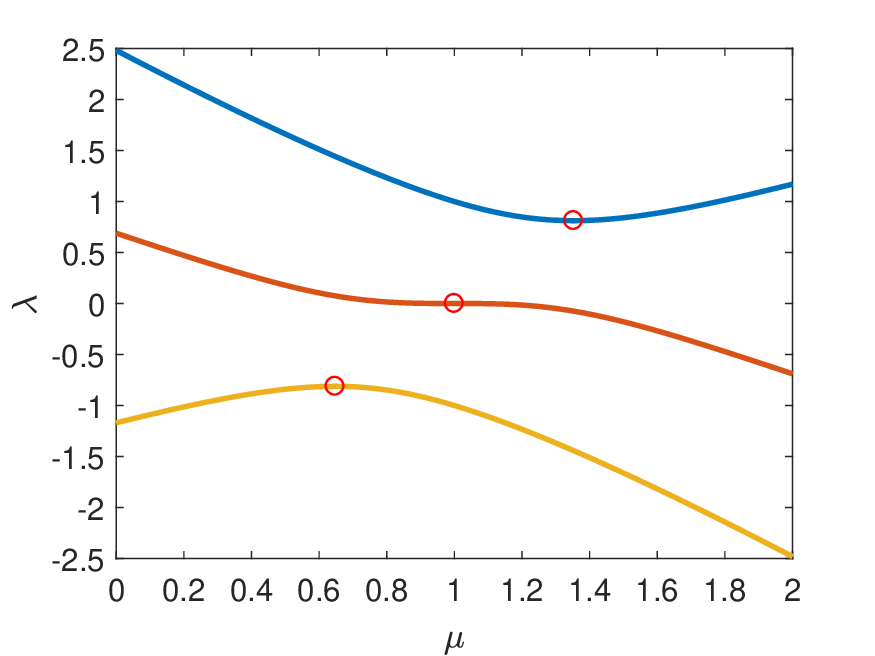}}
\qquad \subfloat[closed up plot of eigencurves $\lambda_2(\mu)$]
{\includegraphics[width=0.45\textwidth]{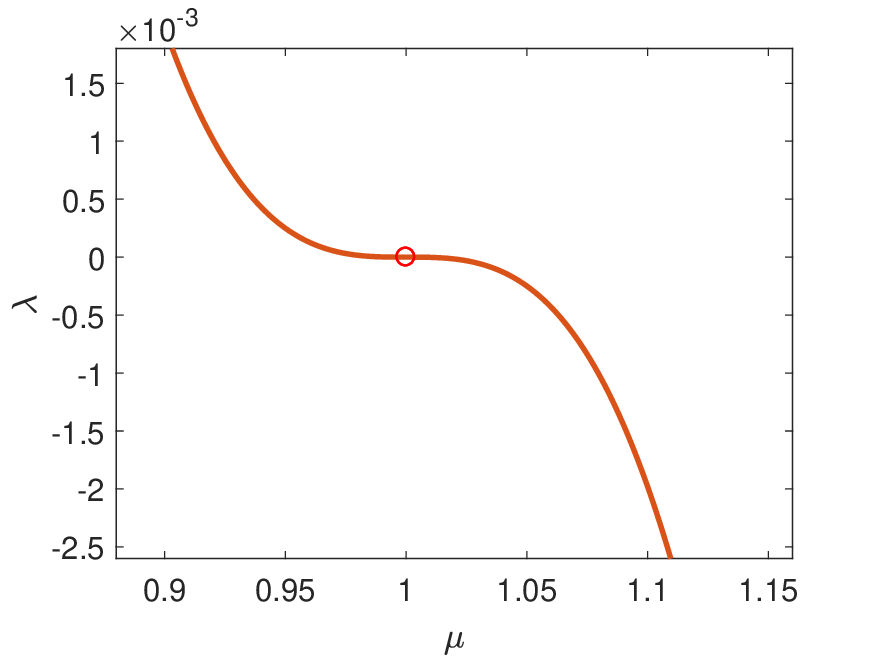}} \\
\subfloat[Three (sorted) eigencurves $\lambda_1(\mu) \geq \lambda_2(\mu) \geq \lambda_3(\mu)$]
{\includegraphics[width=0.45\textwidth]{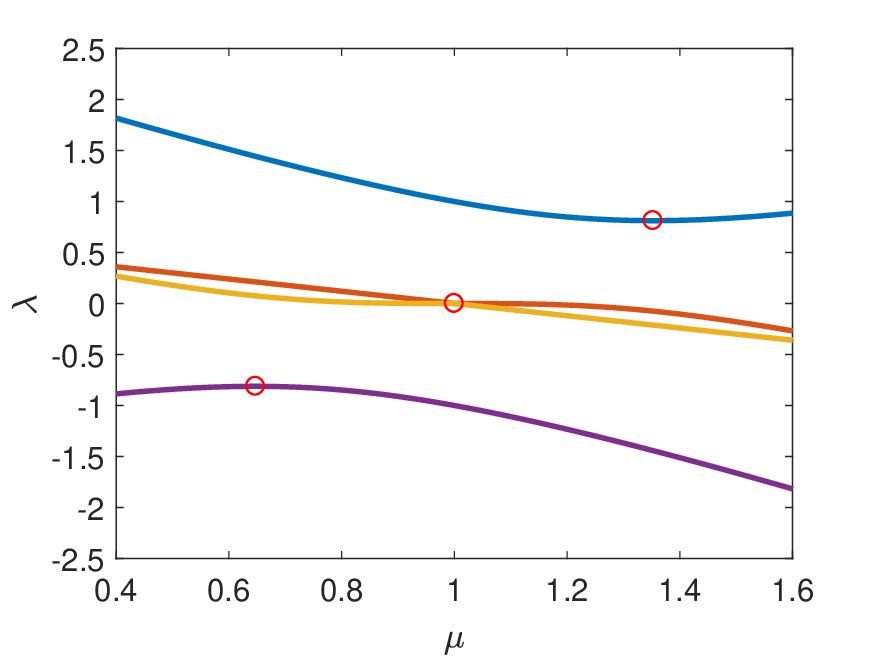}}
\qquad \subfloat[closed up plot of eigencurves $\lambda_2(\mu)$]
{\includegraphics[width=0.45\textwidth]{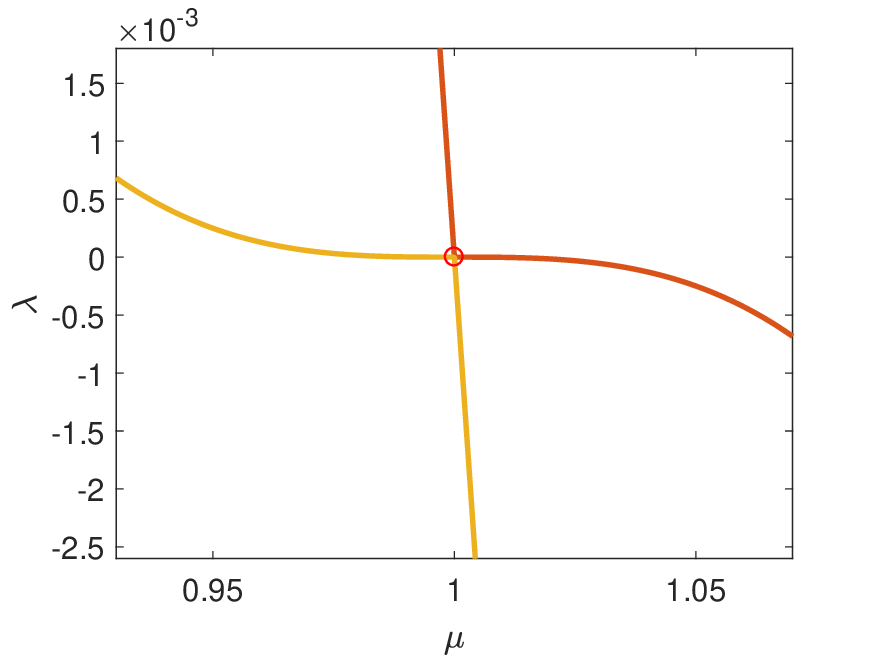}}
\caption{The 2D-eigenvalues can be neither minima nor maxima:
	(a) and (b) for the non-intersection case,
	(c) and (d) for the intersection case.
	}
\label{FigAllevls_singularcase1}
\end{figure}



The following theorem is on the existence of 2D-eigenvalues.

	\begin{theorem}\label{existence}
	The 2DEVP \eqref{2deig} has at least one 2D-eigenvalue.
\end{theorem}
\begin{proof} We prove by construction. 
	Let $\lambda_1(\mu)$ be the largest sorted eigencurve of $A-\mu C$.
	Then as $\mu\rightarrow-\infty$, 
	\[
	\lambda_{1}(\mu) =\max_{\|x\|=1} x^H(A-\mu C)x 
	\geq \lambda_{\min}(A)-\mu\max_{\|x\|=1}x^HCx 
	= \lambda_{\min}(A)-\mu\lambda_C^{(1)}\rightarrow +\infty, 
	\] 
	where $\lambda_C^{(1)}$ is the algebraically largest eigenvalue of $C$. 
	Note that $\lambda_C^{(1)} > 0$ since $C$ is indefinite. 
	On the other hand, as $\mu\rightarrow +\infty$, 
	\[ 
	\lambda_{1}(\mu) \geq \lambda_{\min}(A)-\mu\min_{\|x\|=1}x^HCx
	= \lambda_{\min}(A)-\mu\lambda_C^{(2)}\rightarrow +\infty,
	\] 
	where $\lambda_C^{(2)}$ is the algebraically smallest eigenvalue of $C$. 
	Note that $\lambda_C^{(2)} < 0$.
	Therefore, the minimum of $\lambda_1(\mu)$ is attainable 
	at some point $\mu_*$. By Theorem~\ref{thm_minima_is_2devl}, 
	$(\mu_*, \lambda_1(\mu_*))$ is a 2D-eigenvalue of $(A, C)$.
\end{proof}

	By Theorem~\ref{existence}, we have the following variational 
characterizations of extreme eigencurves 
$\lambda_1(\mu)$ and $\lambda_n(\mu)$.

\begin{theorem}\label{anavarchar}
	Let $\lambda_{1}(\mu)\geq\dots\geq\lambda_{n}(\mu)$ be $n$ sorted eigenvalues
	of $A-\mu C$. Then it holds that
	\begin{equation} \label{eq:var2a} 
		\min_{\mu\in \mathbb{R}} \lambda_1(\mu)= \max_{\substack{x\neq0\\
				x^HCx=0}} \rho_A(x)\quad \mbox{and} \quad \max_{\mu\in \mathbb{R}} \lambda_n(\mu)= \min_{\substack{x\neq0\\
				x^HCx=0}} \rho_A(x).
	\end{equation}
	where $\rho_A(x)$ is the Rayleigh quotient of $A$, 
	$\rho_A(x) = x^H A x /(x^H x)$.
\end{theorem}
\begin{proof}
	We only prove the first identity in \eqref{eq:var2a}. The proof 
	for the second identity is similar. We note that 
	the proof of Theorem~\ref{existence} indicates that the minimum 
	of $\lambda_1(\mu)$ is attainable at some point $\mu_*$ 
	and $(\mu_*,\lambda_*) = (\mu_*,\lambda_1(\mu_*))$ is a 2D-eigenvalue. 
	Let $x_*$ be the corresponding 2D-eigenvector of $(\mu_*, \lambda_*)$, then
	$$
	x_*^HCx_*=0
	\quad \mbox{and} \quad
	\lambda_*=\rho_A(x_*)\leq \max\limits_{\substack{x\neq0\\ x^HCx=0}} \rho_A(x).
	$$
	On the other hand,
	\[
	\lambda_*=\lambda_1(\mu_*)=\max\limits_{x^Hx=1}x^H(A-\mu_* C)x
	\geq\max\limits_{\substack{x^Hx=1\\
			x^HCx=0}}x^H(A-\mu_* C)x\\
	=\max\limits_{\substack{x^Hx=1\\ x^HCx=0}}x^HAx.
	\]
	This completes the proof.
\end{proof}

\medskip
As a corollary of Theorem~\ref{anavarchar}, the following result
provides lower and upper bounds of $\lambda$ 
of 2D-eigenvalues $(\mu,\lambda)$ on the $(\mu,\lambda)$-plane. 

	\begin{corollary}\label{boundfor2devl}
	Let $(\mu_*,\lambda_*)$ be a 2D-eigenvalue of $(A,C)$
	and $\lambda_{1}(\mu)\geq\dots\geq\lambda_{n}(\mu)$ be 
	$n$ sorted eigencurves of $A-\mu C$. Then 
	\begin{equation} \label{eq:lambd}
		\max_{\mu\in \mathbb{R}} \lambda_n(\mu) 
		\le \lambda_* \le
		\min_{\mu\in \mathbb{R}} \lambda_1(\mu),
	\end{equation} 
	where the first equality holds if $\lambda_n(\mu_*) = \lambda_*$,
	and the second equality holds if $\lambda_1(\mu_*)= \lambda_*$. 
\end{corollary}
	\begin{proof}
	Let $x_*$ be the 2D-eigenvector associated with $(\mu_*, \lambda_*)$. 
	Then the inequalities in \eqref{eq:lambd} hold 
	by Theorem~\ref{anavarchar} and the identity $\lambda_*=\rho_A(x_*)$. 
	The equalities hold due to the facts 
	\[
	\max_{\mu\in \mathbb{R}} \lambda_n(\mu) 
	\leq \lambda_*=\lambda_n(\mu_*)
	\leq\max_{\mu\in \mathbb{R}}\lambda_n(\mu)
	\quad \mbox{and} \quad 
	\min_{\mu\in \mathbb{R}} \lambda_1(\mu) 
	\leq \lambda_1(\mu_*) = \lambda_* 
	\leq \min_{\mu\in \mathbb{R}}\lambda_1(\mu).
	\]
\end{proof}

By Corollary~\ref{boundfor2devl},
we have the following definition. 
\begin{definition} \label{def:maxmin2dev} 
Let $\lambda_1(\mu^*)=\min\limits_{\mu\in \mathbb{R}} \lambda_1(\mu)$ and 
$\lambda_n(\mu_*)=\max\limits_{\mu\in \mathbb{R}} \lambda_n(\mu)$. 
Then $(\mu^*, \lambda_1(\mu^*))$ and $(\mu_*, \lambda_n(\mu_*))$ are 
called the maximum and minimum 2D-eigenvalues of $(A, C)$, respectively.
\end{definition}

\medskip
The following theorem provides an upper bound of $|\mu|$ of 
2D-eigenvalues $(\mu,\lambda)$ on the $(\mu,\lambda)$-plane when $C$ 
is nonsingular.

\begin{theorem}\label{boundonmu}
Assume $C$ is nonsingular. Let $\lambda_C^{(+)}$ and $\lambda_C^{(-)}$ be
the minimum positive and maximum negative eigenvalues of $C$, respectively.
If $(\mu_*,\lambda_*,x_*)$ is a 2D-eigentriplet,
then 
$$
|\mu_*|\leq {\|A\|}/{\sqrt{-\lambda_C^{(-)} \lambda_C^{(+)}}}.
$$
\end{theorem}
\begin{proof}
	By multiplying $x^{H}_*C$ on the left of \eqref{eq:1a}, we have 
	$$
	|\mu_*|=\frac{|x^{H}_*CAx_*|}{\|Cx_*\|^2} 
	\leq \frac{\|A\|\|x_*\|}{\|Cx_*\|} = \frac{\|A\|}{\|Cx_*\|}.
	$$
	Then an upper bound of $\frac{\|A\|}{\|Cx_*\|}$ is from a lower bound
	of $\|Cx_*\|$, which leads to compute the quantity 
	\begin{equation} \label{eq:maxC}
		\min\limits_{\substack{x^Hx=1\\ x^HCx=0}}\|Cx\|.
	\end{equation}
	By substituting $C^2$ for $A$ in the second equation of \eqref{eq:var2a}, 
	we have
	\begin{equation}\label{eq:newmuProb}
		\min\limits_{\substack{x^Hx=1\\ x^HCx=0}}\|Cx\|^2
		=\min\limits_{\substack{x^HCx=0\\ x\neq0}}\frac{x^HC^2x}{x^Hx}
		=\max\limits_{\mu\in\mathbb{R}}\lambda_n(C^2-\mu C)
		=\max\limits_{\mu\in\mathbb{R}}\min\{c_i^2-\mu c_i|\ i=1,\cdots,n\},
	\end{equation}
	where $c_1, c_2, \ldots, c_n$ are eigenvalues of $C$. 
	Let $\lambda_C^{(-)} = c_j$ and $\lambda_C^{(+)} = c_k$ for some $j$ and $k$. 
	Then at the intersection 
	$\widetilde{\mu}_*=c_j+c_k$ of lines $c_j^2-\mu c_j$ and $c_k^2-\mu c_k$,
	we have
	\[
	\begin{aligned}
		\max\limits_{\mu\in\mathbb{R}}\min\{c_i^2-\mu c_i|\ i=1,\cdots,n\}&\leq  \max\limits_{\mu\in\mathbb{R}}\min\{c_i^2-\mu c_i|\ i=j,k\}\\
		&=\min\{c_i^2-\widetilde{\mu}_*c_i|\ i=j,k\}
		=-\lambda_C^{(-)}\lambda_C^{(+)}.
	\end{aligned}
	\]
	On the other hand, we can prove 
	$-\lambda_C^{(-)}\lambda_C^{(+)}\leq c_i^2-\tilde{\mu}_*c_i$ 
	for $i=1,\cdots,n$. Without loss of generality, 
	we only consider the case $c_i>0$. Then
	\begin{equation}
		-\lambda_C^{(-)}\lambda_C^{(+)}-(c_i^2-\widetilde{\mu}_*c_i)
		=c_k^2-\widetilde{\mu}_*c_k-c_i^2+\tilde{\mu}_*c_i
		=(c_k-c_i)(c_k+c_i-\widetilde{\mu}_*)
		=(c_k-c_i)(c_i-c_j) \leq 0,
	\end{equation}
	where the first equation is due to the fact 
	$-\lambda_C^{(-)}\lambda_C^{(+)}=c_j^2-\widetilde{\mu}_*c_j
	=c_k^2-\widetilde{\mu}_*c_k$
	and the last inequality results from the fact that 
	either $c_j\leq c_k \leq c_i$ or $c_i\leq c_j\leq c_k$ holds. 
	Hence we have
	\begin{equation*}
		\begin{aligned}
			\max\limits_{\mu\in\mathbb{R}}\min\{c_i^2-\mu c_i|\ i=1,\cdots,n\}
			&\geq  \min\{c_i^2-\tilde{\mu}_* c_i|\ i=1,\cdots,n\}
			=-\lambda_C^{(-)}\lambda_C^{(+)}.
		\end{aligned}
	\end{equation*}
	This implies 
	\[
	\min\limits_{\substack{x^Hx=1\\ x^HCx=0}}\|Cx\|^2
	=\max\limits_{\mu\in\mathbb{R}}\min\{c_i^2-\mu c_i|\ i=1,\cdots,n\}
	=-\lambda_C^{(-)}\lambda_C^{(+)}. 
	\] 
	This completes the proof. 
\end{proof}


\section{Number of 2D-eigenvalues} \label{sec_wellposed}
In this section, we derive a sufficient and necessary condition
for the 2DEVP \eqref{2deig} to have a finite number of 2D-eigenvalues.
By speaking the number of 2D-eigenvalues, we mean the number 
of distinct 2D-eigenvalues.


\begin{theorem} \label{thm:finitemany}
The 2DEVP \eqref{2deig} has a finite number of 2D-eigenvalues
if and only if the matrix pair $(A-\sigma I, C)$ is regular
for any $\sigma \in \mathbb{R}$.\footnote{
A matrix pair $(A, B)$ is called regular if 
$\mbox{det}(A-\theta B) \not\equiv 0$ for any $\theta \in \mathbb{C}$.
Otherwise, it is called singular.}  
\end{theorem}

Geometrically, the sufficient and necessary condition 
in Theorem~\ref{thm:finitemany} is equivalent to 
the fact that there is no any eigencurve $\lambda(\mu)$ of $A - \mu C$ 
which is a constant (a horizontal line) with respective to $\mu$.
The proof of Theorem~\ref{thm:finitemany} is built on the following
four propositions.  The first proposition states a case where 
there are infinitely many 2D-eigenvalues.

\begin{proposition}\label{irregularmeansinfinite}
If there exists a shift $\sigma_0 \in \mathbb{R}$, such that
the matrix pair $(A-\sigma_0I, C)$ is singular,
then the 2DEVP \eqref{2deig} has infinitely many 2D-eigenvalues.  \end{proposition}
\begin{proof}
Let $\widetilde{\lambda}_1(\mu)$, $\cdots$, $\widetilde{\lambda}_n(\mu)$
be the analyticalized eigencurves of $A-\mu C$.
Let $\mu_k=1-\frac{1}{k}$ for $k = 1, 2, \ldots$. 
For each $k$, since
the matrix pair $(A-\sigma_0I, C)$ is singular,
$\det(A-\sigma_0I-\mu_kC)=0$.
Hence $\sigma_0$ is an eigenvalue of $A-\mu_kC$.
Thus there exists at least one $j_k$ such that $\widetilde{\lambda}_{j_k}(\mu_k)=\sigma_0$.
By selecting subsequences, we can assume $j_k$ are the same for each $k$. Without loss of generality, assume $j_k=1$.
Then by Theorem~\ref{analytic_properties}\eqref{item:uniquess}, we have
	$\widetilde{\lambda}_1(\mu)\equiv\sigma_0$.
This implies that $\widetilde{\lambda}_1'(\mu)=0$. By taking derivation on the equation
$(A-\mu C-\widetilde{\lambda}_1(\mu))\widetilde{x}_1(\mu)=0$
and multiplying $\widetilde{x}^H_1(\mu)$ on the left, we obtain
	$0=\widetilde{\lambda}_1'(\mu)=-\widetilde{x}^H_1(\mu)C\widetilde{x}_1(\mu)$.
Thus for any $\mu$, $(\mu,\sigma_0)$ corresponds to a 2D-eigenvalue.
\end{proof}

Here is an example to illustrate Proposition~\ref{irregularmeansinfinite}.

\begin{example}\label{example2}
{\rm
Consider the matrix pair
\[
A=\begin{bmatrix}1&2&0\\2&1&0\\0&0&4\end{bmatrix}
\quad \mbox{and} \quad
C=\begin{bmatrix}1&1&0\\1&1&0\\0&0&-1\end{bmatrix}.
\]
By the shift $\sigma_0 = -1$,
the matrix pair $(A+I, C)$ is singular since
the vector $x=\begin{bmatrix} 1 & -1 & 0 \end{bmatrix}^T$
is in the common nullspace of $A+I$ and $C$.
It can be verified that
$(\mu, -1)$ is a 2D-eigenvalue of $(A,C)$ for any $\mu\in\mathbb{R}$.
Therefore the 2DEVP of $(A,C)$ has infinitely many 2D-eigenvalues. 
$\lambda(\mu) \equiv -1$ is a horizontal eigencurve.
\hfill $\Box$
}\end{example}

We next prove a proposition which says that on a finite interval $[a,b]$,
the number of 2D-eigenvalues $(\mu, \lambda)$ with $\mu\in[a,b]$
is finite if there exists no such a shift as described
in Proposition~\ref{irregularmeansinfinite}.

\begin{proposition}\label{boundedimplyfinite}
If the matrix pair $(A-\sigma I, C)$ is regular
for any $\sigma \in \mathbb{R}$,
then the number of 2D-eigenvalues $(\mu, \lambda)$ with
$\mu\in[a,b]$ is finite, where $a, b \in \mathbb{R}$ are finite.
\end{proposition}
\begin{proof}
Let $\widetilde{\lambda}_k(\mu), k=1,\cdots,n$ be the
analyticalized eigencurves of $A-\mu C$.
According to Corollary~\ref{boundfor2devl}, $\lambda$ of the 2D-eigenvalues $(\mu, \lambda)$
are bounded and we denote the bound by $[\text{lb}, \text{ub}]$. Assume there exist
infinitely many 2D-eigenvalues $(\mu_k, \lambda_k)$ in the close domain
$[a,b]\times[\text{lb}, \text{ub}]$ with $(\mu_k, \lambda_k)\neq(\mu_j, \lambda_j)$ for $j\neq k$, then they must have a convergent subseries. Without loss of generality, we still denote them by $(\mu_k, \lambda_k)$ and assume $(\mu_k, \lambda_k)\rightarrow (\mu^*, \lambda^*)$.

If there are infinitely many 2D-eigenvalues that correspond to intersections, then there exist two analyticalized eigencurves, which for convenience we assume to be $\widetilde{\lambda}_1$, $\widetilde{\lambda}_2$, and a subseries of $\mu_k$, which we still denote as $\mu_k$, such that $\widetilde{\lambda}_1(\mu_k)=\widetilde{\lambda}_2(\mu_k)$ and $\widetilde{\lambda}_1'(\mu_k)\widetilde{\lambda}_2'(\mu_k)\leq0$. Since $\{\mu_k\}_{k=1}^{\infty}$ converges, using Theorem~\ref{analytic_properties}\eqref{item:uniquess} we know $\widetilde{\lambda}_1(\mu)=\widetilde{\lambda}_2(\mu)$, which further implies $\widetilde{\lambda}_1'(\mu)=\widetilde{\lambda}_2'(\mu)$. Thus $\widetilde{\lambda}_1'(\mu_k)=\widetilde{\lambda}_2'(\mu_k)=0$. Using Theorem~\ref{analytic_properties}\eqref{item:uniquess} again, we obtain $\widetilde{\lambda}_1(\mu)=\widetilde{\lambda}_2(\mu)=\lambda^*$. So $\rank \left(A-\lambda^*I-\mu C\right)\leq n-2$ for any $\mu\in\mathbb{R}$. Note that the singularity of the matrix pair $(A-\lambda^*I, C)$ can
be equivalently described as
$\det(A-\lambda^*I-\mu C)=0,\ \forall \mu\in\mathbb{R}$. Thus it contradicts the assumption that $(A-\lambda^*I, C)$ is regular.

Thus we can assume all the $(\mu_k, \lambda_k)$ correspond to non-intersections. There exist one $\widetilde{\lambda}_i$, which for convenience we assume to be $\widetilde{\lambda}_1$, such that $\widetilde{\lambda}_1(\mu_k)=\lambda_k$. Then $\widetilde{\lambda}_1'(\mu_k)=0$ since it is an eigentriplet corresponding to non-intersection. Using Theorem~\ref{analytic_properties}\eqref{item:uniquess} we have $\widetilde{\lambda}_1'(\mu)=0$ and $\widetilde{\lambda}_1(\mu)=\lambda^*$. Thus $\rank \left(A-\lambda^*I-\mu C\right)= n-1$ for any $\mu$ and contradicts the assumption.
\end{proof}

The next proposition shows that under
the conditions of Proposition~\ref{boundedimplyfinite},
all analyticalized eigencurves 
$\widetilde{\lambda}_{\ell}(\mu)$ of $A - \mu C$
are strictly monotonous for sufficiently large $\mu$.

\begin{proposition}\label{monotonous}
If the matrix pair $(A-\sigma I, C)$ is regular
for any shift $\sigma \in \mathbb{R}$,
then for any analyticalized eigencurve
$\widetilde{\lambda}_{\ell}(\mu)$ of $A - \mu C$,
there exists a positive constant $M$ such that
$\widetilde{\lambda}'_{\ell}(\mu)\neq0$ for $|\mu|>M$.
\end{proposition}
\begin{proof}
Let $\widetilde{\lambda}_{\ell}(\mu)=\widetilde{\lambda}_{\ell}(A-\mu C)$ 
be an analyticalized eigencurve. For $\epsilon\neq0$, define 
\begin{equation}\label{eq:defanalyex}
g(\epsilon)\equiv\epsilon\widetilde{\lambda}_{\ell}(\frac{1}{\epsilon})=\epsilon\widetilde{\lambda}_{\ell}(A-\frac{1}{\epsilon} C).
\end{equation}
We note that 
\begin{equation}
\det\left(-C+\epsilon A-g(\epsilon)I\right) = \epsilon^n\det\left(A-\frac{1}{\epsilon}C-\widetilde{\lambda}_{\ell}(A-\frac{1}{\epsilon}C)I\right) = 0.
\end{equation}
Thus $g(\epsilon)$ is an eigenvalue of $-C+\epsilon A$. 
According to \eqref{eq:defanalyex} and Theorem~\ref{analytic_properties}\eqref{item:product}\eqref{item:component}, $g(\cdot)$ is real analytic 
on $(-\infty,0)$ and $(0,+\infty)$. Therefore it must identically equal to
two analyticalized eigencurves of the parameter eigenvalue problem $(-C+\epsilon A)x=\lambda x$, say $\widehat{\lambda}_1(\epsilon)$ and $\widehat{\lambda}_2(\epsilon)$, for $\epsilon<0$ and $\epsilon>0$, respectively.  
 
For $\mu=\frac{1}{\epsilon}\neq0$, we have
\begin{equation}\label{eq:dtildelambda}
\frac{\rm d}{{\rm d}\mu}\widetilde{\lambda}_{\ell}(\mu)
=\frac{\rm d}{{\rm d}\mu}\left(\mu g(\frac{1}{\mu})\right)
= g(\frac{1}{\mu}) - \frac{1}{\mu} g'(\frac{1}{\mu})
= g(\epsilon) - \epsilon g'(\epsilon).
\end{equation}

We first consider the case where $\mu<0$. In such case, \eqref{eq:dtildelambda} becomes 
\[
\frac{\rm d}{{\rm d}\mu}\widetilde{\lambda}_{\ell}(\mu)
= \widehat{\lambda}_1(\epsilon) - \epsilon \widehat{\lambda}_1'(\epsilon),\quad \forall\mu<0.
\]

Since $\widehat{\lambda}_1(\epsilon)$ is real analytic on $\mathbb{R}$, 
$\widehat{\lambda}_1(\epsilon) - \epsilon \widehat{\lambda}_1'(\epsilon)$ is also real analytic on $\mathbb{R}$. 
Thus we have the following two cases: 
\begin{enumerate}[(i)]
\item\label{item:lamprime1} 
$\widehat{\lambda}_1(\epsilon)- \epsilon \widehat{\lambda}_1'(\epsilon)\equiv 0$ on $\mathbb{R}$.

\item\label{item:lamprime2} 
$\widehat{\lambda}_1(\epsilon)- \epsilon \widehat{\lambda}_1'(\epsilon)$ has only finite number 
of roots on any finite closed interval.
\end{enumerate}
If Case~\eqref{item:lamprime1} holds, then the real analytic function 
$\frac{\rm d}{{\rm d}\mu}\widetilde{\lambda}_{\ell}(A-\mu C)=0$ for $\mu\neq0$,
which further implies 
$\frac{\rm d}{{\rm d}\mu}\widetilde{\lambda}_{\ell}(A-\mu C)\equiv0$ 
for $\mu\in\mathbb{R}$. Thus there exists a fixed $\sigma_0$ 
such that $\widetilde{\lambda}_{\ell}(A-\mu C)\equiv\sigma_0$. 
Then $(A-\sigma_0 I, C)$ is singular and it contradicts the assumption.
	
Therefore, only Case~\eqref{item:lamprime2} holds. If 
$\widehat{\lambda}_1(\epsilon)- \epsilon \widehat{\lambda}_1'(\epsilon)$ has no negative roots, 
we define $\epsilon_0=1$. Otherwise we define 
\[
\epsilon_0 = \min\{|\epsilon| \mid 
\widehat{\lambda}_1(\epsilon)- \epsilon \widehat{\lambda}_1'(\epsilon)=0, \epsilon<0\}.
\]
Then we have
$\widehat{\lambda}_1(\epsilon) - \epsilon \widehat{\lambda}_1'(\epsilon)\neq0$
for $\epsilon\in(-\epsilon_0,0)$, 
which is equivalent to
$\frac{\rm d}{{\rm d}\mu}\widetilde{\lambda}_{\ell}(A-\mu C)\neq0$ 
for $\mu\in\left(-\infty,\frac{1}{\epsilon_0}\right)$. Similarily, we could prove there exists $\epsilon_1>0$, such that
$\frac{\rm d}{{\rm d}\mu}\widetilde{\lambda}_{\ell}(A-\mu C)\neq0$ 
for $\mu\in\left(\frac{1}{\epsilon_1},+\infty\right)$. 

Let $M=\frac{1}{\min\{\epsilon_0,\epsilon_1\}}$ and we reach the conclusion.
\end{proof}

By Proposition~\ref{monotonous}, we have the following proposition.
\begin{proposition}\label{notripleoutside}
	If the matrix pair $(A-\sigma I, C)$ is regular
	for any $\sigma \in \mathbb{R}$,
	then there exists a positive constant $\widetilde{M}$
	such that all 2D-eigenvalues $(\mu, \lambda)$ of $(A,C)$ are
	bounded by $[-\widetilde{M},\widetilde{M}] \times \mathbb{R}$.
\end{proposition}
\begin{proof}
Let $\widetilde{\lambda}_1(\mu),\cdots, \widetilde{\lambda}_n(\mu)$ 
be $n$ analyticalized eigencurves. 
According to Proposition~\ref{monotonous}, there exists $M$ such that 
$\widetilde{\lambda}_i(\mu)$ are all monotonous with nonzero derivatives 
when $|\mu|\geq M$. Thus no 2D-eigentriplets corresponding to 
non-intersections exist for $|\mu|\geq M$. If there are infinitely 
many eigentriplets for $|\mu|>M$ corresponding to intersections, 
then there exist two eigencurves having infinitely many intersections. 
However, since they are strictly monotonous for $|\mu|>M$, this cannot happen.
Thus there exist only finitely many 2D-eigentriplets outside 
$\left[-M,M\right]\times [\text{lb},\text{ub}]$. 
$\widetilde{M}$ can be found by increasing $M$.
\end{proof}

{\em Proof of Theorem~\ref{thm:finitemany}.}
The necessary condition is immediately from 
Proposition~\ref{irregularmeansinfinite}.
We only need to prove the sufficiency.
With Proposition~\ref{notripleoutside}, all 2D-eigenvalues are
bounded by $[-\widetilde{M},\widetilde{M}] \times \mathbb{R}$.
Then by Proposition~\ref{boundedimplyfinite},
the total number of 2D-eigenvalues is finite. \hfill $\Box$

\medskip
We now provide a vivid description of
numbers of 2D-eigenvalues on an analyticalized eigencurve.
We will see that besides possible trival 2D-eigenvalues,
there are only finite number of 2D-eigenvalues.
\begin{theorem}\label{finiteandinfinite}
There are only following two cases on  an analyticalized eigencurve
$\widetilde{\lambda}(\mu)$ of $A-\mu C$:
\begin{itemize}
\item there are finite number of 2D-eigenvalues on
$\widetilde{\lambda}(\mu)$, or
\item $\widetilde{\lambda}(\mu)$ is a horizontal line and
all points on the horizontal line are 2D-eigenvalues.
\end{itemize}
\end{theorem}

\begin{proof}
We only need to show that there are only finite number of 
2D-eigenvalues on analyticalized eigencurves that are not horizontal lines. 
For these non-horizontal analyticalized eigencurves, the proof of 
Proposition~\ref{monotonous} actually shows that their derivatives 
are non-zero for sufficiently large $|\mu|$. Utilizing this fact 
and following the proof in Proposition~\ref{notripleoutside}, 
we know all 2D-eigenvalues on them are bounded by 
$[-\widetilde{M},\widetilde{M}]\times\mathbb{R}$ for a large scalar 
$\widetilde{M}$. Finally, the proof in 
Proposition~\ref{boundedimplyfinite} indicates there could not be 
infinite number of 2D-eigenvalues on these non-horizontal analyticalized 
eigencurves among $[-\widetilde{M},\widetilde{M}]\times\mathbb{R}$. 
Thus we reach the conclusion.
\end{proof}

\begin{remark} 
If $C$ is nonsingular, then the matrix pair $(A-\sigma I, C)$ is 
regular for any shift $\sigma\in\mathbb{R}$. This implies that
there are only finite number of 2D eigenvalues of $(A, C)$ 
when $C$ is nonsingular.
\end{remark}

\section{Applications revisited}\label{sec_examples}
In this section, we revisit the two origins of the 2DEVP 
presented in Section~\ref{applications}.

\subsection{Minmax of Rayleigh Quotients}

We first show that the search interval of 
the EVopt~\eqref{eq:EVopt0} can be reduced
from $\mathbb{R}$ to the interval $[0,1]$.

\begin{theorem}\label{thm:app11}
For Case-\ref{item:A3} of Theorem~\ref{thm:classification0}, 
i.e., $\lambda_A\leq\theta_B$ and $\lambda_B\leq\theta_A$,
the EVopt~\eqref{eq:EVopt0} satisfies
\begin{equation}\label{firststatement}
\max\limits_{\mu\in\mathbb{R}}\lambda_{\min}(A-\mu C)
= \max\limits_{{\mu\in[0,1]}}\lambda_{\min}(A-\mu C).
\end{equation}
Furthermore, if $\lambda_A<\theta_B$ and $\lambda_B<\theta_A$, 
the EVopt~\eqref{eq:EVopt0} satisfies 
\begin{equation}\label{secondstatement}
\arg\max\limits_{\mu\in\mathbb{R}}\lambda_{\min}(A-\mu C)
= \arg\max\limits_{{\mu\in(0,1)}}\lambda_{\min}(A-\mu C).
\end{equation}
\end{theorem}
\begin{proof}
By Theorem~\ref{thm_convex}, 
the minimum eigenvalue $\lambda_{\min}(\mu) \equiv \lambda_{n}(\mu)$ 
of $A-\mu C$ is concave.
Thus to prove the identity \eqref{firststatement}, it is sufficient 
to prove that 
\begin{equation}\label{eq:QCQP1}
\lambda^{'(-)}_{n}(1)\equiv 
\lim\limits_{t\rightarrow 0^+}\frac{\lambda_n(1)-\lambda_n(1-t)}{t} \leq 0,
\end{equation}
and 
\begin{equation}\label{eq:QCQP2}
\lambda^{'(+)}_{n}(0)\equiv 
\lim\limits_{t\rightarrow 0^+}\frac{\lambda_n(t)-\lambda_n(0)}{t} \geq 0.
\end{equation}
For \eqref{eq:QCQP1}, if $\lambda_n(1)=\lambda_B$ is 
a simple eigenvalue of $A-C=B$, then $\lambda_n$ is differentiable 
at $\mu=1$ and $\lambda_n'(1) = -x^HCx = -x^H(A-B)x$,
where $x$ is the eigenvector of $B$ associated with $\lambda_B$. Thus 
$\lambda_n'(1) = -x^HAx+x^HBx =\lambda_B-\theta_A \leq 0$
and \eqref{eq:QCQP1} holds.
If $\lambda_n(1)$ is not simple, by Theorem~\ref{thm:onesidedderivative}, $\lambda^{'(-)}_{n}(1)$ equals to the maximum eigenvalue of $-C_k = {S_B^H(B-A)S_B}$, where $S_B$ is the eigen-subspace of $\lambda_n(1)=\lambda_B$ for $A-C = B$. In 
Case-\ref{item:A3} of Theorem~\ref{thm:classification0}, 
any nonzero $x$ belonging to $S_B$ satisfies 
$x^HBx=\lambda_B\leq\theta_A\leq x^HAx$. Thus $-C_k$ is 
negative or semi-negative definite. This is to say, 
$\lambda^{'(-)}_{n}(1)\leq0$. The argument for \eqref{eq:QCQP2} is similar. 
Thus the identity~\eqref{firststatement} holds. 

To prove \eqref{secondstatement}, we only need to prove the inequalities 
in \eqref{eq:QCQP1} and \eqref{eq:QCQP2} are strict. For the strict inequality 
in \eqref{eq:QCQP1}, it is sufficient to prove $\lambda^{'}_{n}(1)<0$ 
when multiplicity of $\lambda_B$ is one and 
$\lambda^{'(-)}_{n}(1)<0$ when multiplicity of $\lambda_B$ is larger than one. 
This can be proved using the same argument and noting $\lambda_B<\theta_A$. 
The argument for the strict inequality in $\eqref{eq:QCQP2}$ is similar.
\end{proof}

By combining Theorems~\ref{thm:classification0} and \eqref{thm:app11}, 
we establish the following equivalence between the RQminmax 
and the eigenvalue optimization:    
\begin{equation}\label{yuanlemma}
\min\limits_{x\neq 0}\max\left\{\frac{x^TAx}{x^Tx}, \frac{x^TBx}{x^Tx}\right\}
= \max\limits_{\mu\in[0,1]}\lambda_{\min}(A-\mu C).
\end{equation}
Specifically, if it is Case-\ref{item:A1} of
Theorem~\ref{thm:classification0}, by using the similar argument 
in the proof of the identity \eqref{firststatement},
in Theorem~\ref{thm:app11},
we can prove that $\lambda^{'(+)}_{n}(0)<0$. Thus with the concavity 
of $\lambda_n(\mu)$, for any $\mu \in [0,1]$, we have
\[
\lambda_n(\mu)\leq \lambda_n(0) + \mu\lambda^{'(+)}_n(0)\leq 
\lambda_n(0) = \lambda_{\min}(A).
\]
This implies that the equation \eqref{yuanlemma} holds 
in Case-\ref{item:A1}. The similar argument shows the 
equation \eqref{yuanlemma} also holds for Case-\ref{item:A2}.

We note that the equivalence identity \eqref{yuanlemma} is exactly 
Yuan's lemma \cite[Lemma~2.3]{Yuan1990}, which is closely related 
to the well-known S-lemma in control theory and robust 
optimization~\cite{Polik2007,zong2010}.

By Definition~\ref{def:maxmin2dev}, we know that if $\mu_*$ is
an optimizer of the eigenvalue optimization, 
then $(\mu_*, \lambda_{\min}(\mu_*))$ is the minimum 2D-eigenvalue
of $(A, C)$. We will demonstrate the advantages of the 2DEVP 
reformulation of the RQminmax in our forthcoming work on algorithms 
for solving the 2DEVP. 

\subsection{Distance to instability}
From Section~\ref{applications}, we learn that
the calculation of the distance to instability $\beta(\widehat{A})$ 
can be recast as an eigenvalue optimization problem \eqref{eq_distance}, 
namely $\beta(\widehat{A}) = \lambda_m(\mu_*)$, where
$\mu_*$ is an optimizer of the $m$-th eigencurve $\lambda_m(\mu)$ of $A-\mu C$.
Van Loan \cite{ahownear1985} proved that the optimizer 
$\mu_* \in [-2\|{A}\|, 2\|{A}\|]$.

By Theorem~\ref{thm_minima_is_2devl}, if $\mu_*$ is an optimizer of \eqref{eq_distance}, then 
$(\mu_*, \beta(\widehat{A}))$ is a 2D-eigenvalue of the 2DEVP of $(A, C)$.
In addition, we have the following list of characterizations of 
the target 2D-eigentriplet:
\begin{itemize} 
	\item If $(\mu,\lambda,\left[\begin{smallmatrix}
		u\\ v
	\end{smallmatrix}\right])$ is a 2D-eigentriplet of $(A, C)$, 
	then $(\mu,-\lambda,\left[\begin{smallmatrix}
		-u\\ v \end{smallmatrix}\right])$ is also a 2D-eigentriplet. 
	This implies the 2D-eigenvalues are symmetric with regard to $\lambda = 0$.
	
	\item Based on the ordering of $2m$ eigenvalues of $A-\mu C$: 
	\begin{equation}\label{eq:sortedEvlInDTI}
		\lambda_1(\mu) \ge \lambda_2(\mu) \ge\cdots\ge
		\lambda_m(\mu) >0>\lambda_{m+1}(\mu)  \ge\cdots\ge \lambda_{2m}(\mu),
	\end{equation}
	we have the following characterization of $\beta(\widehat{A})$:
	\begin{equation}\label{eq:locationDTI0}
		\begin{aligned}
			\beta(\widehat{A}) &=	\min\{\lambda \mid 
			\mbox{$(\mu,\lambda)$ is a 2D-eigenvalue of $(A,C)$ and $\lambda>0$}
			\}\\
			&=-\max\{\lambda \mid 
			\mbox{$(\mu,\lambda)$ is a 2D-eigenvalue of $(A,C)$ and $\lambda<0$} \}\\
			&=\min\{|\lambda| \mid 
			\mbox{$(\mu,\lambda)$ is a 2D-eigenvalue of $(A,C)$}
			\}.
		\end{aligned}
	\end{equation}
\end{itemize} 
Moreover, by Theorem~\ref{boundonmu} on the range of $\mu$ of
the 2D-eigenvalue $(\mu,\lambda)$,  
we immediately conclude that the optimal $\mu_*$ must be in 
$[-\|{A}\|, \|{A}\|]$. This shortens the search interval by a half. 
Finally, by Theorem~\ref{thm:finitemany}, we know that 
there is only a finite number of local minima 
of $\lambda_m(\mu)$ in $[-\|{A}\|, \|{A}\|]$. 

\begin{example}{\rm 
Consider the following stable matrix
from \cite[Example~5]{Freitag2011Aaa}:
\[
\widehat{A} = \begin{bmatrix}
          -0.4+6{\tt i} & 1 &  &  \\
	  1 & -0.1+1{\tt i} & 1 &  \\
	   & 1 & -1-3{\tt i} & 1 \\
	   &  & 1 & -5+1{\tt i}
        \end{bmatrix}.
\]
The distance to instability is
\[
\beta(\widehat{A}) = \min_{\mu \in \mathbb{R}} \lambda_{4}(A - \mu C)
\quad \mbox{with} \,
A = \begin{bmatrix} & \widehat{A} \\ \widehat{A}^H & \end{bmatrix}
\, \mbox{and} \,
B = \begin{bmatrix} & {\tt i}I_4  \\ -{\tt i}I_4 & \end{bmatrix}.
\]
Figure~\ref{sigmaminplot} shows the eigencurve $\lambda_{4}(\mu)$
on the interval $[-2\|A\|, 2\|A\|]$.
As we can see that $\lambda_{4}(\mu)$ is monotonic outside
the interval $[-\|A\|, \|A\|]$. The optimal $\mu_*$ locates within
$[-\|A\|, \|A\|]$ masked by "*".
\hfill $\Box$

\begin{figure}[tbhp]
\centering
\includegraphics[scale=0.45]{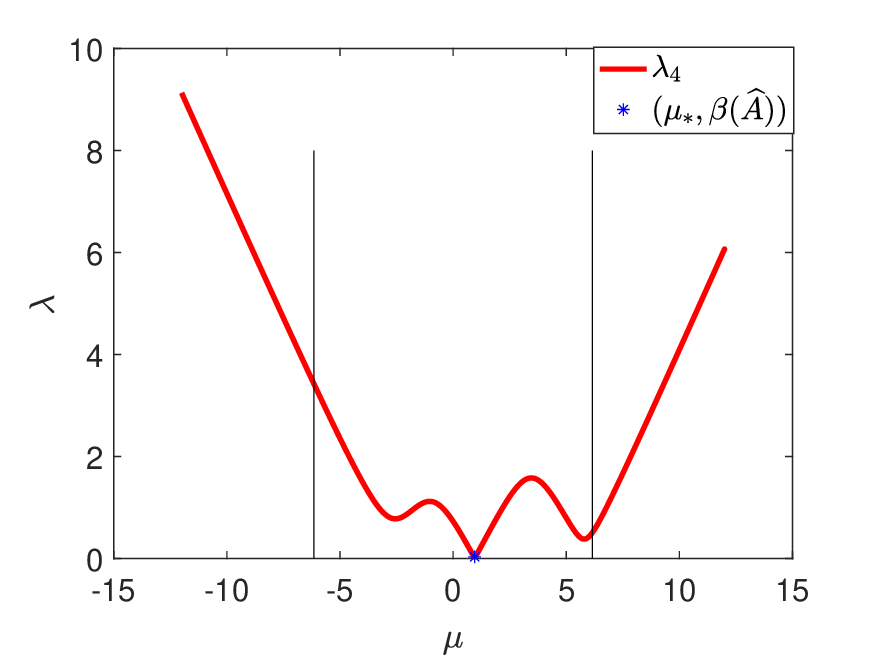}
\caption{$\lambda_{4}(\mu)$ on the interval $[-2\|A\|, 2\|A\|]$.
The vertical lines are $\pm\|A\|$.
Note that $\lambda_{4}(\mu)$ is monotonic outside 
the interval $[-\|A\|, \|A\|]$.
}
\label{sigmaminplot}
\end{figure}

}\end{example} 

Numerous algorithms have been developed for computing the
distance $\beta(\widehat{A})$
\cite{boyd1990,bBye1988,Freitag2011Aaa,dinverse1999,csparse2006,ahownear1985}.
By reformulating the distance problem to the 2DEVP, it provides
an opportunity for a new class of methods to solve the distance problem,
in particular, for handling large scale ones.
This will be the subject in the sequel of this paper.


\section{Conclusion}  \label{sec:conclude}
We introduced a 2DEVP of a Hermitian matrix pair $(A,C)$.
We highlighted the relationship between
the well-known eigenvalue optimization problem of the parameter matrix
$H(\mu) = A - \mu C$ and the 2DEVP.
We presented essential properties of the 2DEVP such as the
existence and the necessary and sufficient condition for
the finite number of 2D-eigenvalues.
In addition, we also provided the variational characterizations of
2D-eigenvalues and the bound of 2D-eigenvalues.
We used eigenvalue optimization problems from the
quadratic constrained quadratic program and the computation
of distance to instability to show new insights of these problems
derived from the properties of the 2DEVP.

This paper is the first in a sequel on the study of the 2DEVP.
An immediate sequel to this work will focus on numerical algorithms for
solving the 2DEVP.


\appendix 
\section{Proof of Theorem~\ref{thm:classification0}} \label{append:thmRQ} 


For Case-\ref{item:A1}, we first note that for any $x\neq0$, 
	\[
	\max\{\rho_A(x),\rho_B(x)\} \geq \rho_A(x)\geq\lambda_A.
	\]
	On the other hand, for $x_* = S_A z_B$,  
	\[
	\max\{\rho_A(x_*),\rho_B(x_*)\}
	= \max\{z^H_B S^H_A A S_A z_B,  z^H_B S^H_A B S_A z_B \} 
	= \max\{\lambda_A, \theta_B\} = \lambda_A, 
	\] 
	where for the last equality, we use the condition 
	{$\lambda_A>\theta_B$}.
	Thus $x_*$ is a solution of the RQminmax~\eqref{prob:minmaxRQ}. 
	
Case-\ref{item:A2} is proven by exchanging the roles
	of $A$ and $B$ in the proof of Case-\ref{item:A1}.
	
	For Case-\ref{item:A3} 
	(i.e., {$\lambda_A \leq \theta_B$ and $\lambda_B \leq \theta_A$}), 
	we first prove by contradiction that in this case, 
	if $x_*$ is a solution of the RQminmax~\eqref{prob:minmaxRQ}, then 
	\begin{equation}\label{eq:xCx}
	x_*^HCx_*=0. 
	\end{equation}

We prove by contradiction. Assume $x_*^HCx_*>0$, 
i.e., $x_*^HAx_*>x_*^HBx_*$. Then $x_*$ does not belong to $S_A$ 
since otherwise $\lambda_A = \rho_A(x_*)>\rho_B(x_*)\geq\theta_B$ 
and contradicts the condition that $\lambda_A\leq\theta_B$. 
Consider $x(t) = x_*+t\sign(x_A^Hx_*)x_A$ with $t>0$, 
where by convention $\sign(0)=1$. A straightforward calculation shows
\begin{equation}
	\rho_A(x(t)) = \frac{\|x_*\|^2\rho_A(x_*)+
		\left(t^2 \|x_A\|^2+2t |x_*^Hx_A|\right)\lambda_A}{\|x_*\|^2+t^2 \|x_A\|^2
		+2t |x_*^Hx_A|}<\rho_A(x_*).
\end{equation}
On the other hand, by the continuity of $\rho_A(x(t))$ and 
$\rho_B(x(t))$ 
with respect to $t$, $\rho_B(x(t))<\rho_A(x(t))$ holds for 
a sufficiently small $t$. This implies for such $t$ we have 
\[
\max\{\rho_A(x(t)), \rho_B(x(t))\} 
= \rho_A(x(t)) < \rho_A(x_*) 
= \max\{\rho_A(x_*), \rho_B(x_*)\}, 
\]
which contradicts the condition that $x_*$ is the solution of 
the RQminmax~\eqref{prob:minmaxRQ}. Hence $x_*^HCx_*\leq0$. 
A similar argument 
leads to $x_*^HCx_*\geq0$ and therefore we have the identity \eqref{eq:xCx}. 
	
	The identity \eqref{eq:xCx} implies that in Case-\ref{item:A3}, we have
	\[
	\arg\min\limits_{x\neq0}\max\{\rho_A(x),\rho_B(x)\}=\arg\min\limits_{\substack{x^HCx=0\\x\neq 0}}\max\{\rho_A(x),\rho_B(x)\}=\arg\min\limits_{\substack{x^HCx=0\\x\neq 0}}\rho_A(x).\]
	Thus $x_*$ is the optimizer of the RQminmax~\eqref{prob:minmaxRQ} 
	if and only if $x_*$ is the optimizer of the 
	following constrained Rayleigh quotient optimization problem:
	\begin{equation} \label{primalproblem1}
	\min_{\substack{x^HCx=0\\x\neq 0}} \rho_A(x).
	\end{equation}
	Now according to Theorem~\ref{anavarchar} and the definition of $\mu_*$, we have
	\begin{equation}\label{eq:temp2}
	\min\limits_{\substack{x^HCx=0\\x\neq 0}} \rho_A(x) 
	= \max\limits_{\mu\in \mathbb{R}}\lambda_{\min}(A-\mu C) 
	= \lambda_{\min}(A-\mu_*C).
	\end{equation}
	Thus, by \eqref{primalproblem1} and \eqref{eq:temp2}, we have    
	\begin{align*}
	\arg\min_{x\neq 0}\max\{\rho_A(x),\rho_B(x)\}
	& =\arg\min_{\substack{x^HCx=0\\x\neq 0}} \rho_A(x)  \\ 
	&=\{x\neq0 \mid \rho_A(x)=\lambda_{\min}(A-\mu_*C), x^HCx=0\}\\
	&=\{x\neq0 \mid \rho_{A-\mu_*C}(x)=\lambda_{\min}(A-\mu_*C), x^HCx=0\},
	\end{align*}
	Note that for any $x\neq0$, 
	the identity $\rho_{A-\mu_*C}(x)=\lambda_{\min}(A-\mu_*C)$ 
	is equivalent to the fact that $x$ is an eigenvector corresponding to 
	$\lambda_{\min}(A-\mu_*C)$. \\	
	This completes the proof that
	for Case-\ref{item:A3},  
	$x_*$ is the solution of 
	the RQminmax~\eqref{prob:minmaxRQ} if and only if 
	$x_*$ is an eigenvector of $A - \mu_* C$
	corresponding to $\lambda_{\min}(A-\mu_* C)$ and $x_*^HCx_*=0$. 
	\hfill $\Box$



\bibliographystyle{siamplain}
\bibliography{2devp}

\begin{thebibliography}{10}

\bibitem{BC1978}
{\sc E.~K. Blum and A.~F. Chang}, {\em A numerical method for the solution of
  double eigenvalue problem}, J. Inst. Math. Appl., 22 (1978), pp.~29--42,
  \url{https://doi.org/10.1093/imamat/22.1.29}.

\bibitem{boyd1990}
{\sc S.~Boyd and V.~Balakrishnan}, {\em A regularity result for the singular
  values of a transfer matrix and a quadratically convergent algorithm for
  computing its ${L}_{\infty}$ norm}, Systems Control Lett., 15 (1990),
  pp.~1--7, \url{https://doi.org/10.1109/CDC.1989.70267}.

\bibitem{bBye1988}
{\sc R.~Byers}, {\em A bisection method for measuring the distance of a stable
  matrix to the unstable matrices}, SIAM J. Sci. Statist. Comput., 9 (1988),
  pp.~875--881, \url{https://doi.org/10.1137/0909059}.

\bibitem{Celis1985}
{\sc M.~Celis, J.~Dennis, and R.~Tapia}, {\em A trust region strategy for
  nonlinear equality constrained optimization}, Numerical Optimization,
  (1985), p.~71–82.

\bibitem{Clarke1990}
{\sc F.~H. Clarke}, {\em Optimization and Nonsmooth Analysis}, SIAM,
  Philadelphia, PA, 1990, \url{https://doi.org/10.1137/1.9781611971309}.

\bibitem{demmel1997applied}
{\sc J.~W. Demmel}, {\em Applied Numerical Linear Algebra}, SIAM, Philadelphia,
  PA, 1997, \url{https://doi.org/10.1137/1.9781611971446}.

\bibitem{Fan1949}
{\sc K.~Fan}, {\em On a theorem of {W}eyl concerning eigenvalues of linear
  transformations {I}}, Proc. Nat. Acad. Sci., U.S.A., 35 (1949), pp.~652--655,
  \url{https://doi.org/10.1073/pnas.35.11.652}.

\bibitem{FN1995}
{\sc M.~K.~H. Fan and B.~Nekooie}, {\em On minimizing the largest eigenvalue of
  a symmetric matrix}, Linear Alg. Appl., 214 (1995), pp.~225--246,
  \url{https://doi.org/10.1016/0024-3795(93)00068-B}.

\bibitem{Freitag2011Aaa}
{\sc M.~A. Freitag and A.~Spence}, {\em A {N}ewton-based method for the
  calculation of the distance to instability}, Linear Algebra Appl., 435
  (2011), pp.~3189--3205, \url{https://doi.org/10.1016/j.laa.2011.06.012}.

\bibitem{GH2013}
{\sc D.~D. Gaurav and K.~V.~S. Hari}, {\em A fast eigen solution for
  homogeneous quadratic minimization with at most three constraints}, IEEE
  Signal Process. Lett., 20 (2013), pp.~968--971,
  \url{https://doi.org/10.1109/LSP.2013.2276791}.

\bibitem{Gershman2010}
{\sc A.~B. Gershman, N.~D. Sidiropoulos, S.~Shahbazpanahi, M.~Bengtsson, and
  B.~Ottersten}, {\em Convex optimization-based beamforming}, IEEE Signal
  Process. Mag., 27 (2010), pp.~62--75,
  \url{https://doi.org/10.1109/MSP.2010.936015}.

\bibitem{Gohberg2009Matrix}
{\sc I.~Gohberg, P.~Lancaster, and L.~Rodman}, {\em Matrix Polynomials}, SIAM,
  Philadelphia, PA., 2009, \url{https://doi.org/10.1137/1.9780898719024}.

\bibitem{dinverse1999}
{\sc C.~He and G.~A. Watson}, {\em An algorithm for computing the distance to
  instability}, SIAM J. Matrix Anal. Appl., 20 (1999), pp.~101--116,
  \url{https://doi.org/10.1137/S0895479897314838}.

\bibitem{Kangal2018}
{\sc F.~Kangal, K.~Meerbergen, E.~Mengi, and W.~Michiels}, {\em A subspace
  method for large-scale eigenvalue optimization}, SIAM J. Matrix Anal. Appl.,
  39 (2018), pp.~48--82, \url{https://doi.org/10.1137/16M1070025}.

\bibitem{Karipidis2007}
{\sc E.~Karipidis, N.~D. Sidiropoulos, and Z.~Luo}, {\em Far-field multicast
  beamforming for uniform linear antenna arrays}, IEEE Trans. Signal Process.,
  55 (2007), pp.~4916--4927, \url{https://doi.org/10.1109/TSP.2007.897903}.

\bibitem{Kha2005}
{\sc V.~B. Khazanov}, {\em Methods for solving spectral problems for
  multiparameter matrix pencils}, J. Math. Sciences, 127 (2005),
  pp.~2033--2050, \url{https://doi.org/10.1007/s10958-005-0161-8}.

\bibitem{Knopp1990theoryI}
{\sc K.~Knopp}, {\em Theory of {F}unctions (Part I)}, Dover, New York, 1947.

\bibitem{Krantz2002}
{\sc H.~R.~P. Krantz, S.~G.}, {\em A {P}rimer of {R}eal {A}nalytic
  {F}unctions}, Birkh\"{a}user, Boston, MA, 2~ed., 2002.

\bibitem{csparse2006}
{\sc D.~Kressner}, {\em Finding the distance to instability of a large sparse
  matrix}, in Proce. {IEEE} {I}nternational {S}ymposium on {I}ntelligent
  {C}ontrol, Munich, 2006, pp.~31--35,
  \url{https://doi.org/10.1109/CACSD-CCA-ISIC.2006.4776620}.

\bibitem{KLV2017}
{\sc D.~Kressner, D.~Lu, and B.~Vandereycken}, {\em Subspace acceleration for
  the {Crawford} number and related eigenvalue optimization problems}, SIAM J.
  Matrix Anal. Appl., 39 (2018), pp.~961--982,
  \url{https://doi.org/10.1137/17M1127545}.

\bibitem{Lewis1996}
{\sc A.~S. Lewis}, {\em Derivatives of spectral functions}, Math. Oper. Res.,
  21 (1996), pp.~576--588, \url{https://doi.org/10.1287/moor.21.3.576}.

\bibitem{LO1996}
{\sc A.~S. Lewis and M.~L. Overton}, {\em Eigenvalue optimization}, Acta
  Numerica, 5 (1996), pp.~149--190,
  \url{https://doi.org/10.1017/S0962492900002646}.

\bibitem{MYK2014}
{\sc E.~Mengi, E.~A. Yildirim, and M.~Kilic}, {\em Numerical optimization of
  eigenvalues of {Hermitian} matrix functions}, SIAM J. Matrix Anal. Appl., 35
  (2014), pp.~699--724, \url{https://doi.org/10.1137/130933472}.

\bibitem{Ove1988}
{\sc M.~L. Overton}, {\em On minimizing the maximum eigenvalue of a symmetric
  matrix}, SIAM J. Matrix Anal. Appl., 9 (1988), pp.~256--268,
  \url{https://doi.org/10.1137/0609021}.

\bibitem{Overton1991}
{\sc M.~L. Overton}, {\em Large-scale optimization of eigenvalues}, SIAM J.
  Optim., 2 (1992), pp.~88--120, \url{https://doi.org/10.1137/0802007}.

\bibitem{Overton1995}
{\sc M.~L. Overton and R.~S. Womersley}, {\em Second derivatives for optimizing
  eigenvalues of symmetric matrices}, SIAM J. Matrix Anal. Appl., 16 (1995),
  pp.~697--718, \url{https://doi.org/10.1137/S089547989324598X}.

\bibitem{Polik2007}
{\sc I.~P$\acute{o}$lik and T.~Terlaky}, {\em A survey of the {S}-lemma}, SIAM
  Rev., 49 (2007), pp.~371--418,
  \url{https://doi.org/10.1137/S003614450444614X}.

\bibitem{POLAK1982}
{\sc E.~Polak and Y.~Wardi}, {\em Nondifferentiable optimization algorithm for
  designing control systems having singular value inequalities}, Automatica, 18
  (1982), pp.~267--283, \url{https://doi.org/10.1016/0005-1098(82)90087-5}.

\bibitem{Rel1969}
{\sc F.~Rellich}, {\em Perturbation Theory of Eigenvalue Problems}, Gordon and
  Breach Science Publishers, New York, 1969.

\bibitem{Roos1999}
{\sc C.~Roos, T.~Terlaky, A.~Nemirovski, and K.~Roos}, {\em On maximization of
  quadratic form over intersection of ellipsoids with common center}, Math.
  Program., 86 (1999), pp.~463--473,
  \url{https://doi.org/10.1007/s101070050100}.

\bibitem{SP2005}
{\sc A.~Spence and C.~Poulton}, {\em Photonic band structure calculations using
  nonlinear eigenvalue techniques}, J. Comp. Physics, 204 (2005), pp.~65--81,
  \url{https://doi.org/10.1016/j.jcp.2004.09.016}.

\bibitem{Trefethen2005}
{\sc L.~Trefethen and M.~Embree}, {\em Spectra and {P}seudospectra: the
  {B}ehavior of {N}onnormal {M}atrices and {O}perators}, Princeton University
  Press, NJ, 2005.

\bibitem{ahownear1985}
{\sc C.~F. Van~Loan}, {\em How near is a stable matrix to an unstable matrix?},
  Contemp. Math., 47 (1985), pp.~465--478,
  \url{https://doi.org/10.1090/conm/047/828319}.

\bibitem{zong2010}
{\sc Z.-Z. Yan and J.~Guo}, {\em Some equivalent results with {Y}akubovich's
  {S}-lemma}, SIAM J. Control Optim., 48 (2010), pp.~4474--4480,
  \url{https://doi.org/10.1137/080744219}.

\bibitem{Yuan1990}
{\sc Y.~Yuan}, {\em On a subproblem of trust region algorithms for constrained
  optimization}, Math. Program., 47 (1990), pp.~53--63,
  \url{https://doi.org/10.1007/BF01580852}.

\bibitem{Zhang2009}
{\sc R.~Zhang, Y.~Liang, C.~C. Chai, and S.~Cui}, {\em Optimal beamforming for
  two-way multi-antenna relay channel with analogue network coding}, IEEE J.
  Sel. Areas Commun., 27 (2009), pp.~699--712,
  \url{https://doi.org/10.1109/JSAC.2009.090611}.

\bibitem{Zhang2011}
{\sc Y.~J.~A. Zhang and A.~M. So}, {\em Optimal spectrum sharing in {MIMO}
  cognitive radio networks via semidefinite programming}, IEEE J. Sel. Areas
  Commun., 29 (2011), pp.~362--373,
  \url{https://doi.org/10.1109/JSAC.2011.110209}.

\end{thebibliography}
\end{document}